\documentclass[amstex,11pt,reqno,english]{amsart}
\usepackage{amsmath,amsfonts,amssymb,amsthm,enumerate,hyperref,multicol,graphicx}
\usepackage{sidecap}
\usepackage{lineno,hyperref}
\usepackage{mathrsfs,  amssymb, bbm}
\usepackage{wrapfig}
\usepackage{babel,blindtext}
\usepackage{caption, float}
\usepackage{subfig}
\usepackage{amsmath}
\newtheorem{lemma}{Lemma}
\newtheorem{thm}{Theorem}

\newtheorem{remark}{Remark}
\numberwithin{equation}{section}
\numberwithin{thm}{section}
\numberwithin{corollary}{section}
\numberwithin{remark}{section}
\numberwithin{lemma}{section}

\usepackage{mathtools}

\allowdisplaybreaks
\begin{document}
\title[Milstein Scheme for SDEwMS]{A Note on Explicit Milstein-Type Scheme for Stochastic Differential Equation with Markovian Switching}

\author{Chaman Kumar}
\address{Chaman Kumar, Department of Mathematics\\
Indian Institute of Technology Roorkee, India}
\email{C.Kumarfma@iitr.ac.in}

\author{Tejinder Kumar}
\address{Tejinder Kumar, Department of Mathematics\\
Indian Institute of Technology Roorkee, India}
\email{tejinder.dma2017@iitr.ac.in}

\maketitle

\begin{abstract}
An explicit Milstein-type scheme for stochastic differential equation with Markovian switching is derived and its strong convergence in $\mathcal{L}^2$-sense is established without using  It\^o-Taylor expansion formula. Rate of strong convergence is shown to be equal to $1.0$ under the assumptions that coefficients satisfy mild regularity conditions. More precisely, coefficients are assumed to be only once differentiable which are more relaxed conditions than those made in existing literature.
\end{abstract}

\section{\bf{Introduction}}

Stochastic differential equation with Markovian switching (SDEwMS) has found several applications in real world such as   \cite{bao2016, mao2006, sethi1994,  yin2010, zhang1998, zhang2001} and references therein. Often, explicit solution of SDEwMS is not available and hence one requires numerical approximation for such equation. The order $1/2$  Euler scheme for SDEwMS has been discussed in literature for example  \cite{nguyen2018, nguyen2012, yang2018} and references therein. Recently, a Milstein-type scheme for SDEwMS has been developed in \cite{nguyen2017}. Authors in \cite{nguyen2017} give an  It\^o's formula (see e.g. Lemma 2.2) for the switching coefficient and hence derive a Milstein-type scheme for SDEwMS. Their approach for derivation and for establishing strong convergence results (in $\mathcal{L}^2$-sense) of Milstein-type scheme is inspired by   well known It\^o-Taylor expansion and hence  authors impose second order differentiability assumptions on the coefficients.  Motivated by \cite{kumar2019}, a new approach for  derivation and for establishing strong convergence results (in $\mathcal{L}^2$-sense) of Milstein-type scheme is  developed in this article that do not  require It\^o-Taylor expansion. As a consequence of our approach, drift and diffusion coefficients are assumed to be only once differentiable and hence is a significant reduction on regularity requirements on the coefficients when compared with the corresponding results obtained in \cite{nguyen2017}. 

Let us now introduce some notations used in this article. For $b\in \mathbb{R}^d$ and $\sigma \in \mathbb{R}^{d\times m}$, $|b|$ and $|\sigma|$ are used for Euclidean and Hilbert-schmidt norms  respectively which is clear from the context at which they appear.  The $l$-th element of a vector $b\in\mathbb{R}^d$ is denoted by $b_l$ and the $l$-th column of a matrix $\sigma\in\mathbb{R}^{d\times m}$ is denoted by $\sigma_{(l)}$. For $x,y\in\mathbb{R}^d$, $xy$ stands for  their inner product.  Further, if $f:\mathbb{R}^d \to \mathbb{R}^d$, then $\mathcal{D}f$ returns a $d\times d$ matrix with  $\frac{\partial f_i(\cdot)}{\partial x_j}$ as $(i,j)$-th entry for  $i,j=1,\ldots,d$. All through this article, $C>0$ stands for a generic constant which can vary from place to place and is always independent of the discretization step-size. 

\section{\bf{Main Result}}

\noindent
Let $ \big( \Omega ,   \mathscr{F}, P \big)$  be  a complete probability space. Suppose that $W:=\{W(t); t\geq0\}$ is  an $\mathbb{R}^m$-valued standard Wiener process. Also, assume that $\alpha:=\{\alpha(t);t\geq 0\}$ is a continuous-time Markov chain with  finite state space $\mathcal{S}:=\{1,2,\ldots,m_0 \}$,  for a fixed positive integer $m_0$. The local behaviour of the chain is governed by the generator $Q=(q_{i_0j_0}; i_0, j_0 \in \mathcal{S})$ with $q_{i_0j_0}\geq 0$,  for any $i_0\neq j_0\in \mathcal{S}$ and $q_{i_0i_0}=-\sum_{j_0\neq i_0} q_{i_0j_0}$ for any $i_0\in\mathcal{S}$. Further, let $b:\mathbb{R}^d\times \mathcal{S} \to \mathbb{R}^d$ and $\sigma:\mathbb{R}^d\times \mathcal{S} \to \mathbb{R}^{d \times m}$ be functions satisfying certain conditions to be specified later. The main aim of this article is to study the Milstein scheme for the following $d$-dimensional Stochastic Differential Equation with Markovian Switching (SDEwMS), 
\begin{align} \label{eq:sdems}
dX(t)=b(X(t),\alpha(t))dt+\sigma(X(t),\alpha(t))dW(t)
\end{align} 
almost surely for any $t\in[0,T]$ with initial value $X(0)$, which is an $\mathscr{F}_0$-measurable random variable taking values in $\mathbb{R}^d$. It is further assumed that  $X(0)$,  $W$ and $\alpha$ are independent. Also, let   $\mathscr{F}^W$ and $\mathscr{F}^\alpha$ be filtrations generated by $(X(0),W)$ and $\alpha$ respectively  \textit{i.e.} $\mathscr{F}_t^W:=\sigma\{X_0,\, W(s);\, 0\leq s\leq t\}$ and $\mathscr{F}_t^\alpha:=\sigma\{\alpha(s);\, 0\leq s\leq t\}$. Define $\mathscr{F}_t:=\mathscr{F}_t^W \vee \mathscr{F}_t^\alpha$ for any $t\geq 0$.

Let us now introduce the  Milstein scheme of SDEwMS \eqref{eq:sdems}. For this, one partitions the interval $[0,T]$ into subintervals of equal length $h>0$ \textit{i.e.}, $t_n=nh$ for any $n=0,1,\ldots, n_T$ with $n_T:=Th^{-1}$ and  define $\Delta_n t:=h=t_{n+1}-t_n$, $\Delta_n W:=W(t_{n+1})-W(t_n)$ for $n=0,1,\ldots,n_T-1$.  The Milstein scheme for SDEwMS \eqref{eq:sdems} at grid point $t_{n+1}$ is given by,
\begin{align} \label{eq:scheme}
Y_{n+1} & =Y_n+b(Y_n, \alpha_n) h +\sum_{l=1}^m\sigma_{(l)}(Y_n,\alpha_n)\Delta_nW_l\notag
\\
&+\sum_{l,l_1=1}^m\int_{t_n}^{t_{n+1}}  \int_{t_n}^{s} \mathcal{D} \sigma_{(l)}(Y_n,\alpha_n) \sigma_{(l_1)}(Y_n,\alpha_n) dW_{l_1}(u) dW_l(s) \notag 
\\
&+\sum_{l=1}^m \mathbbm{1}_{\{N_n=1\}}\Big(\sigma_{(l)}(Y_n,\alpha_{n+1})-\sigma_{(l)}(Y_n,\alpha_n)\Big)\Big(W_l(t_{n+1})-W_l(\tau_1^n)\Big)
\end{align}
almost surely with initial value $Y_0$ which is an $\mathscr{F}_0$-measurable random variable in $\mathbb{R}^d$. Here $\alpha_n=\alpha(t_n)$, $N_n$ is the  number of jumps and $\tau_1^n$ is the time of first jump of the chain $\alpha$ in the interval $(t_n, t_{n+1})$ for any $n=0,1,\ldots,n_T-1$. To show that  Milstein scheme \eqref{eq:scheme} of SDEwMS \eqref{eq:sdems} has a rate of convergence equal to $1$, the following assumptions are made. Let $p\geq 2$ be a fixed constant. 
\newline  
\noindent 
\textbf{Assumption H-1.} There exists a  constant $L>0$ such that  $E|X_0|^{p}\leq L$ and $E|X_0-Y_0|^2\leq L h^2$. 
\newline 
\noindent 
\textbf{Assumption H-2.} 
There exists a  constant $L>0$ such that, for every $i_0\in \mathcal{S}$, 
\begin{align*}
|b(x,i_0)-b(y,i_0)|\vee |\sigma(x,i_0)-\sigma(y,i_0)|\leq L|x-y|
\end{align*}
 for any  $x,y\in \mathbb{R}^d$. 
\newline 
\noindent 
\textbf{Assumption H-3.} There exists a constant $L>0$ such that, for every $i_0\in \mathcal{S}$,
\begin{align*}
|\mathcal{D}b(x,i_0)-\mathcal{D}b(y,i_0)|\vee|\mathcal{D}\sigma_{(l)}(x,i_0)-\mathcal{D}\sigma_{(l)}(y,i_0)|&\leq L|x-y|
\\
|\mathcal{D}\sigma_{(l)}(x,i_0)\sigma_{(l_1)}(x,i_0)-\mathcal{D}\sigma_{(l)}(y,i_0)\sigma_{(l_1)}(y,i_0)| &\leq L|x-y|
\end{align*} 
for any $x,y\in \mathbb{R}^d$ and $ l,l_1=1,\ldots, m$.  
\begin{remark}\label{rem:growth}
Due to Assumptions H-2 and H-3,  for every $i_0\in\mathcal{S}$, 
\begin{align*}
|b(x,i_0)| \vee |\sigma(x,i_0)| & \leq L(1+|x|)
\\
|\mathcal{D}b(x,i_0)|\vee |\mathcal{D} \sigma_{(l)}(x,i_0)| & \leq L
\end{align*}
for any $x\in\mathbb{R}^d$ and and $ l=1,\ldots, m$.  The case when the drift coefficient satisfies one-sided  and polynomial Lipschitz conditions (i.e. have  super-linear growth)  is developed in our joint work \cite{kumar2019a} where we propose a tamed Milstein scheme for SDEwMS.  
\end{remark}
The following is the main result of this article. 
\begin{thm} \label{thm:main}
Let Assumptions H-1, H-2 and H-3 be satisfied. Then, the Milstein scheme \eqref{eq:scheme} converges in $\mathcal{L}^2$-sense to the true solution of SDEwMS \eqref{eq:sdems} with rate of convergence equal to $1$, \textit{i.e.}, there exists a constant $C>0$ independent of $h$ such that, 
$$
E\Big(\sup_{n\in\{0,1,\ldots,n_T\}}|X_n-Y_n|^2 \Big)\leq C h^{2}
$$
where $X_n=X(t_n)$ for any $n=0,1,\ldots,n_T$ and $0<h<1/(2q)$ with $q:=\max\{-q_{i_0i_0};i_0\in\mathcal{S}\}$. 
\end{thm}
\section{\bf{Derivation of  Milstein Scheme.}}
In this section, we explain our ideas of detailed derivation of the Milstein scheme. This forms the motivation for reducing the regularity requirement on the coefficients, which is the main achievement in this article. More precisely, we assume that $b(\cdot,i_0)$ and $\sigma(\cdot,i_0)$ are once differentiable for every $i_0 \in \mathcal{S}$ whereas authors in \cite{nguyen2017} assume that they are twice differentiable. The following new way of deriving the Milstein scheme achieves this objective. First, we define the martingale associated with the chain $\alpha$ as introduced in \cite{nguyen2017}. For each $i_0, j_0 \in \mathcal{S}$, $i_0\neq j_0$, define, 
\begin{align*}
[M_{i_0j_0}](t)&:=\sum_{0\leq s\leq t}\mathbbm{1}_{\{\alpha(s-)=i_0\}} \mathbbm{1}_{\{\alpha(s)=j_0\}}
\\
\langle M_{i_0j_0}\rangle(t)&:=\int_0^t q_{i_0j_0}\mathbbm{1}_{\{\alpha(s-)=i_0\}} ds
\\
M_{i_0j_0}(t)&:= [M_{i_0 j_0}](t)-\langle M_{i_0j_0}\rangle(t)
\end{align*}
almost surely for any $t\in[0,T]$. The processes $\{[M_{i_0 j_0}](t); \, t\in[0,T]\}$ and $\{\langle M_{i_0j_0}\rangle(t); \, t\in[0,T]\}$ are respectively optional and predictable quadratic variations whereas $\{M_{i_0j_0}(t); \, t\in[0,T]\}$ is a purely discontinuous and square integrable martingale with respect to filtration $\{\mathscr{F}_{t}^\alpha; \, t \in [0,T]\}$ with $M_{i_0j_0}(0)=0$ almost surely. For notational convenience, take $M_{i_0i_0}(t)=0$ for any $i_0\in\mathcal{S}$ and $t\in[0,T]$.   First, we prove the following useful lemma.     
\begin{lemma}\label{l for applying mvt}
Let $\tau_1< \tau_2< \ldots< \tau_\nu$ be the times of jumps of the chain $\alpha$ in the interval $(r,t)$ for any $0\leq r<t\leq T$, where $t$ may or may not be the jump time of the chain and $\nu$ depends on $r,t$ i.e. $\nu:=\nu(r,t)$. Define $\tau_0:=r$ and $\tau_{\nu+1}:=t$. Also, suppose that $g(\cdot,i_0):\mathbb{R}^d \to \mathbb{R}^d$ is  function for every $i_0 \in\mathcal{S}$. Then, one has,  
\begin{align*}
g(X(t),\alpha(t))&-g(X(r), \alpha(r))=\sum_{i_0\neq j_0}\int_r^t(g(X(u),j_0)-g(X(u),i_0))dM_{i_0j_0}(u)
\\
&+\sum_{j_0\in\mathcal{S}}\int_r^t q_{\alpha(u-)j_0}\big(g(X(u),j_0)-g(X(u),\alpha(u-))\big)du
\\
&+ \sum_{k =0}^{\nu(r,t)}\Big( g(X(\tau_{k+1}),\alpha(\tau_{k}))-g(X(\tau_{k}),\alpha(\tau_{k}))  \Big)
\end{align*}
almost surely for any $0\leq r<t\leq T$.  
\end{lemma}
\begin{proof}
First, one  writes, 
\begin{align}
g(X(t),\alpha(t))&-g(X(r), \alpha(r))=g(X(\tau_{\nu+1}),\alpha(\tau_{\nu+1}))-g(X(\tau_{0}), \alpha(\tau_{0})) \notag
\\
& =\sum_{k=0}^{\nu}\bigg( g(X(\tau_{k+1}),\alpha(\tau_{k+1}))-g(X(\tau_{k}), \alpha(\tau_{k}))\bigg) \notag
\\
& =\sum_{k=0}^{\nu}\bigg( g(X(\tau_{k+1}),\alpha(\tau_{k+1}))-g(X(\tau_{k+1}),\alpha(\tau_{k}))\bigg) \notag
\\
&\qquad+ \sum_{k=0}^{\nu}\bigg(g(X(\tau_{k+1}),\alpha(\tau_{k})) -g(X(\tau_{k}), \alpha(\tau_{k}))\bigg) \label{eq:g:lem}
\end{align}
almost surely. For the first term on the right hand side of the above equation, one  observes, 
\begin{align*}
\sum_{i_0\neq j_0}\int_r^t & \Big(g(X(u),j_0)-g(X(u),i_0)\Big)dM_{i_0j_0}(u)
\\
& =\sum_{i_0\neq j_0}\int_r^t \Big(g(X(u),j_0)-g(X(u),i_0)\Big)d[ M_{i_0j_0}](u)
\\
&\qquad- \sum_{i_0\neq j_0}\int_r^t \Big(g(X(u),j_0)-g(X(u),i_0)\Big)d\langle M_{i_0j_0}\rangle(u)
\\
&=\sum_{k=0}^{\nu}\bigg( g(X(\tau_{k+1}),\alpha(\tau_{k+1}))-g(X(\tau_{k+1}),\alpha(\tau_{k}))\bigg)
\\
&\qquad - \sum_{i_0\neq j_0} \int_r^t q_{i_0j_0}\mathbbm{1}_{\{\alpha(u-)=i_0\}}\Big(g(X(u),j_0)-g(X(u),i_0)\Big)du
\\
&=\sum_{k=0}^{\nu}\bigg( g(X(\tau_{k+1}),\alpha(\tau_{k+1}))-g(X(\tau_{k+1}),\alpha(\tau_{k}))\bigg)
\\
&\qquad - \sum_{ j_0\in\mathcal{S}} \int_r^t q_{\alpha(u-) j_0}\Big(g(X(u),j_0)-g(X(u),\alpha(u-))\Big)du
\end{align*}
which on using in equation \eqref{eq:g:lem} completes the proof.
\end{proof}
One can now proceed with the derivation of the Milstein scheme \eqref{eq:scheme}. Here, we remark that the following calculations are done carefully without using  It\^o's formula in such a way that the coefficients $b(\cdot,i_0)$ and $\sigma(\cdot, i_0)$ are assumed to be  only once differentiable for every $i_0\in\mathcal{S}$, which are weaker regularity assumptions than those made in \cite{nguyen2017}. Let $N_n$ denotes the number of jumps  and $\tau^n_1<\tau^n_2<\ldots<\tau^n_{N_n}$ be the jump times of the chain $\alpha$ in the interval $(t_n,t_{n+1}]$ for any $n=0,1,\ldots,n_T-1$. For notational convenience, take    $\tau_0^n=t_n$, $\tau_{N_n+1}^n=t_{n+1}$,  $\alpha_n=\alpha(t_n)$ and $X_n=X(t_n)$ for any $n=0,1,\ldots,n_T$. Let us write the SDEwMS \eqref{eq:sdems} in the following form, 
\begin{align}
X_{n+1}=&X_{n}+\int_{t_n}^{t_{n+1}}b(X(s),\alpha(s))ds+\sum_{l=1}^m\int_{t_n}^{t_{n+1}}\sigma_{(l)}(X(s),\alpha(s))dW_l(s)\notag 
\\
=&X_{n}+\int_{t_n}^{t_{n+1}}b(X_n,\alpha_n)ds+\sum_{l=1}^m\int_{t_n}^{t_{n+1}}\sigma_{(l)}(X_n,\alpha_n)dW_l(s)\notag
\\
&+\int_{t_n}^{t_{n+1}}\Big(b(X(s),\alpha(s))-b(X_n,\alpha_n)\Big)ds\notag
\\
&+\sum_{l=1}^m\int_{t_n}^{t_{n+1}}\Big(\sigma_{(l)}(X(s),\alpha(s))-\sigma_{(l)}(X_n,\alpha_n)\Big)dW_l(s)\label{breaking of b and sigma}
\end{align}
almost surely for any $n=0,1,\ldots,n_T-1$. Now, one uses Lemma \ref{l for applying mvt} with $g=b$, $t=s$ and $r=t_n$ to obtain the following, 
\begin{align}
b(X(s),&\alpha(s))-b(X_n,\alpha_n)=\sum_{i_0\neq j_0}\int_{t_n}^s\Big(b(X(u),j_0)-b(X(u),i_0)\Big)dM_{i_0j_0}(u)\notag
\\
&+\sum_{j_0\in\mathcal{S}}\int_{t_n}^s q_{\alpha(u-)j_0}\Big(b(X(u),j_0)-b(X(u),\alpha(u-))\Big)du \notag
\\
&+\sum_{k =0}^{\nu(t_n,s)}\Big( b(X(\tau^n_{k+1}),\alpha(\tau^n_{k}))-b(X(\tau^n_{k}),\alpha(\tau^n_{k}))  \Big)\notag 
\\
=&\sum_{i_0\neq j_0}\int_{t_n}^s\Big(b(X(u),j_0)-b(X(u),i_0)\Big)dM_{i_0j_0}(u)\notag
\\
&+\sum_{j_0\in\mathcal{S}}\int_{t_n}^s q_{\alpha(u-)j_0}\Big(b(X(u),j_0)-b(X(u),\alpha(u-))\Big)du \notag
\\
+ &\sum_{k =0}^{\nu(t_n,s)}\Big( b(X(\tau^n_{k+1}),\alpha(\tau^n_{k}))-b(X(\tau^n_{k}),\alpha(\tau^n_{k})) \notag
\\
&\qquad- \mathcal{D} b(X(\tau^n_{k}),\alpha(\tau^n_{k}))( X(\tau^n_{k+1})-X(\tau^n_{k})) \Big) \notag
\\
&+\sum_{k =0}^{\nu(t_n,s)}\int_{\tau^n_{k}}^{\tau^n_{k+1}}\mathcal{D}b(X(\tau^n_{k}),\alpha(\tau^n_{k}))b(X(u),\alpha(u))du \notag 
\\
&+\sum_{k =0}^{\nu(t_n,s)} \sum_{l=1}^m\int_{\tau^n_{k}}^{\tau^n_{k+1}}\mathcal{D} b(X(\tau^n_{k}),\alpha(\tau^n_{k})) \sigma_{(l)}(X(u),\alpha(u))dW_{l}(u)
\label{b by using lemma of for applying mvt}
\end{align}
and similarly for the last term of equation \eqref{breaking of b and sigma} with $g=\sigma_{(l)}$, $t=s$ and $r=t_n$,   which on substituting in equation \eqref{breaking of b and sigma} yields the following, 
\begin{align}
X_{n+1}&=X_{n}+\int_{t_n}^{t_{n+1}}b(X_n,\alpha_n)ds+\sum_{l=1}^m\int_{t_n}^{t_{n+1}}\sigma_{(l)}(X_n,\alpha_n)dW_l(s)\notag
\\
& + \int_{t_n}^{t_{n+1}} \sum_{i_0\neq j_0}\int_{t_n}^s\Big(b(X(u),j_0)-b(X(u),i_0)\Big)dM_{i_0j_0}(u) ds\notag
\\
&+\int_{t_n}^{t_{n+1}} \sum_{j_0\in\mathcal{S}}\int_{t_n}^s q_{\alpha(u-)j_0}\Big(b(X(u),j_0)-b(X(u),\alpha(u-))\Big)du ds \notag
\\
& + \int_{t_n}^{t_{n+1}} \sum_{k =0}^{\nu(t_n,s)}\Big( b(X(\tau^n_{k+1}),\alpha(\tau^n_{k}))-b(X(\tau^n_{k}),\alpha(\tau^n_{k}))  \notag
\\
&\qquad- \mathcal{D} b(X(\tau^n_{k}),\alpha(\tau^n_{k}))( X(\tau^n_{k+1})-X(\tau^n_{k})) \Big) ds \notag
\\
&+\int_{t_n}^{t_{n+1}} \sum_{k =0}^{\nu(t_n,s)} \int_{\tau^n_{k}}^{\tau^n_{k+1}}\mathcal{D} b(X(\tau^n_{k}),\alpha(\tau^n_{k}))b(X(u),\alpha(u))du ds \notag 
\\
&+\int_{t_n}^{t_{n+1}} \sum_{k =0}^{\nu(t_n,s)}\sum_{l=1}^m\int_{\tau^n_{k}}^{\tau^n_{k+1}}\mathcal{D} b(X(\tau^n_{k}),\alpha(\tau^n_{k})) \sigma_{(l)}(X(u),\alpha(u))dW_{l}(u) ds \notag
\\
& + \sum_{l=1}^m\int_{t_n}^{t_{n+1}} \sum_{i_0\neq j_0}\int_{t_n}^s\Big(\sigma_{(l)}(X(u),j_0)-\sigma_{(l)}(X(u),i_0)\Big)dM_{i_0j_0}(u) dW_l(s) \notag
\\
&+\sum_{l=1}^m\int_{t_n}^{t_{n+1}}  \sum_{j_0\in\mathcal{S}}\int_{t_n}^s q_{\alpha(u-)j_0}\Big(\sigma_{(l)}(X(u),j_0)-\sigma_{(l)} (X(u),\alpha(u-))\Big)du  dW_l(s) \notag
\\
& + \sum_{l=1}^m\int_{t_n}^{t_{n+1}}  \sum_{k =0}^{\nu(t_n,s)}\Big( \sigma_{(l)}(X(\tau^n_{k+1}),\alpha(\tau^n_{k}))-\sigma_{(l)}(X(\tau^n_{k}),\alpha(\tau^n_{k})) \notag
\\
&\qquad -\mathcal{D} \sigma_{(l)}(X(\tau^n_{k}),\alpha(\tau^n_{k}))( X(\tau^n_{k+1})-X(\tau^n_{k})) \Big) dW_l(s) \notag
\\
&+\sum_{l=1}^m\int_{t_n}^{t_{n+1}}  \sum_{k =0}^{\nu(t_n,s)}\int_{\tau^n_{k}}^{\tau^n_{k+1}} \mathcal{D} \sigma_{(l)}(X(\tau^n_{k}),\alpha(\tau^n_{k}))b(X(u),\alpha(u))du dW_l(s) \notag 
\\
&+\sum_{l=1}^m\int_{t_n}^{t_{n+1}}  \sum_{k =0}^{\nu(t_n,s)} \sum_{l_1=1}^m\int_{\tau^n_{k}}^{\tau^n_{k+1}} \Big(\mathcal{D} \sigma_{(l)}(X(\tau^n_{k}),\alpha(\tau^n_{k})) \sigma_{(l_1)}(X(u),\alpha(u)) \notag
\\
&\qquad - \mathcal{D} \sigma_{(l)}(X_n,\alpha_n) \sigma_{(l_1)}(X_n,\alpha_n) \Big) dW_{l_1}(u) dW_l(s) \notag
\\
&+\sum_{l=1}^m\int_{t_n}^{t_{n+1}}  \sum_{k =0}^{\nu(t_n,s)} \sum_{l_1=1}^m\int_{\tau^n_{k}}^{\tau^n_{k+1}} \mathcal{D} \sigma_{(l)}(X_n,\alpha_n)\sigma_{(l_1)}(X_n,\alpha_n) dW_{l_1}(u) dW_l(s) \label{Scheme der. without two terms}
\end{align}
almost surely for any $n=0,1,\ldots,n_T-1$. The last term on the right hand side of  equation \eqref{Scheme der. without two terms} can be written as, 
\begin{align}
\sum_{l=1}^m\int_{t_n}^{t_{n+1}}  \sum_{k =0}^{\nu(t_n,s)} \sum_{l_1=1}^m\int_{\tau^n_{k}}^{\tau^n_{k+1}} \mathcal{D} \sigma_{(l)}(X_n,\alpha_n) \sigma_{(l_1)}(X_n,\alpha_n) dW_{l_1}(u) dW_l(s) \notag
\\
=\sum_{l,l_1=1}^m\int_{t_n}^{t_{n+1}}  \int_{t_n}^{s} \mathcal{D} \sigma_{(l)}(X_n,\alpha_n) \sigma_{(l_1)}(X_n,\alpha_n) dW_{l_1}(u) dW_l(s) \label{eq:last}
\end{align}
for any $n=0,1,\ldots,n_T-1$. Also, the ninth term on the right hand side of equation \eqref{Scheme der. without two terms} can be expressed as follows, 
\begin{align}
& \sum_{l=1}^m \int_{t_n}^{t_{n+1}}\sum_{i_0\neq j_0}\int_{t_n}^s\Big(\sigma_{(l)}(X(u),j_0)-\sigma_{(l)}(X(u),i_0)\Big)dM_{i_0j_0}(u)dW_l(s)\notag
\\
=&\sum_{l=1}^m \int_{t_n}^{t_{n+1}}\sum_{i_0\neq j_0}\int_{t_n}^s\Big(\sigma_{(l)}(X(u),j_0)-\sigma_{(l)}(X(u),i_0)\Big)d[M_{i_0j_0}](u)dW_l(s)\notag
\\
&+\sum_{l=1}^m \int_{t_n}^{t_{n+1}}\sum_{i_0\neq j_0}\int_{t_n}^s\Big(\sigma_{(l)}(X(u),i_0)-\sigma_{(l)}(X(u),j_0)\Big)d\langle M_{i_0j_0}\rangle(u)dW_l(s)\notag
\\
=&\sum_{l=1}^m \mathbbm{1}_{\{N_n=1 \}} \Big(\sigma_{(l)}(X(\tau_1^n),\alpha_{n+1})-\sigma_{(l)}(X(\tau_1^n),\alpha_n)\Big)\Big(W_l(t_{n+1})-W_l(\tau_1^n)\Big)\notag 
\\
&+\sum_{l=1}^m  \mathbbm{1}_{\{N_n\geq 2\}} \int_{t_n}^{t_{n+1}} \sum_{i_0\neq j_0} \int_{t_n}^s\Big(\sigma_{(l)}(X(u),j_0)-\sigma_{(l)}(X(u),i_0)\Big)d[M_{i_0j_0}](u)dW_l(s)\notag
\\
&+\sum_{l=1}^m \int_{t_n}^{t_{n+1}} \sum_{i_0\neq j_0} \int_{t_n}^s\Big(\sigma_{(l)}(X(u),i_0)-\sigma_{(l)}(X(u),j_0)\Big)d\langle M_{i_0j_0}\rangle(u)dW_l(s)\notag
\\
=&\sum_{l=1}^m  \mathbbm{1}_{\{N_n=1 \}} \Big( \Big( \sigma_{(l)}(X(\tau_1^n),\alpha_{n+1})-\sigma_{(l)}(X_n,\alpha_{n+1})\Big)\notag
\\
&\qquad -\Big( \sigma_{(l)}(X(\tau_1^n),\alpha_{n})-\sigma_{(l)}(X_n,\alpha_{n})\Big) \Big) \Big(W_l(t_{n+1})-W_l(\tau_1^n)\Big) \notag
\\
&+\sum_{l=1}^m \mathbbm{1}_{\{N_n=1 \}} \Big(\sigma_{(l)}(X_n,\alpha_{n+1})-\sigma_{(l)}(X_n,\alpha_n)\Big)\Big(W_l(t_{n+1})-W_l(\tau_1^n)\Big)\notag 
\\
&+\sum_{l=1}^m \mathbbm{1}_{\{N_n\geq 2\}} \int_{t_n}^{t_{n+1}} \sum_{i_0\neq j_0} \int_{t_n}^s\Big(\sigma_{(l)}(X(u),j_0)-\sigma_{(l)}(X(u),i_0)\Big)d[M_{i_0j_0}](u)dW_l(s)\notag
\\
&+\sum_{l=1}^m \int_{t_n}^{t_{n+1}} \sum_{i_0\neq j_0} \int_{t_n}^s\Big(\sigma_{(l)}(X(u),i_0)-\sigma_{(l)}(X(u),j_0)\Big)d\langle M_{i_0j_0}\rangle(u)dW_l(s) \label{scheme last term expression}
\end{align}
almost surely for any $n=0,1,\ldots, n_T-1$. On substituting values from equations \eqref{eq:last} and \eqref{scheme last term expression} in equation \eqref{Scheme der. without two terms}, one obtains, 
\begin{align}
&X_{n+1}=X_{n}+b(X_n,\alpha_n)h+\sum_{l=1}^m \sigma_{(l)}(X_n,\alpha_n)\Delta_n W_l\notag
\\
&+\sum_{l,l_1=1}^m\int_{t_n}^{t_{n+1}}  \int_{t_n}^{s} \mathcal{D} \sigma_{(l)}(X_n,\alpha_n) \sigma_{(l_1)}(X_n,\alpha_n) dW_{l_1}(u) dW_l(s) \notag
\\
&+\sum_{l=1}^m\mathbbm{1}_{\{N_n=1 \}}\Big(\sigma_{(l)}(X_n,\alpha_{n+1})-\sigma_{(l)}(X_n,\alpha_n)\Big)\Big(W_l(t_{n+1})-W_l(\tau_1^n)\Big)  +\sum_{i=1}^{12}R_{n}(i) \label{scheme complete derivation} 
\end{align}
almost surely for any $n=0,1,\ldots, n_T-1$. The Milstein scheme \eqref{eq:scheme} is constructed from the above equation by ignoring the remainder terms  $R_{n}(i)$ for $i=1,\ldots,12$. The remainder terms in the above equation \eqref{scheme complete derivation} are defined as below,  
\begin{align*}
R_n(1)&:=\int_{t_n}^{t_{n+1}} \sum_{i_0\neq j_0}\int_{t_n}^s\Big(b(X(u),j_0)-b(X(u),i_0)\Big)dM_{i_0j_0}(u) ds 
\\
R_n(2)&:=\int_{t_n}^{t_{n+1}} \sum_{j_0\in\mathcal{S}}\int_{t_n}^s q_{\alpha(u-)j_0}\Big(b(X(u),j_0)-b(X(u),\alpha(u-))\Big)du ds 
\\
R_n(3)&:=\int_{t_n}^{t_{n+1}} \sum_{k =0}^{\nu(t_n,s)}\Big( b(X(\tau^n_{k+1}),\alpha(\tau^n_{k}))-b(X(\tau^n_{k}),\alpha(\tau^n_{k})) \notag
\\
&\qquad - \mathcal{D} b(X(\tau^n_{k}),\alpha(\tau^n_{k}))( X(\tau^n_{k+1})-X(\tau^n_{k})) \Big) ds 
\\
R_n(4)&:=\int_{t_n}^{t_{n+1}} \sum_{k =0}^{\nu(t_n,s)} \int_{\tau^n_{k}}^{\tau^n_{k+1}}\mathcal{D} b(X(\tau^n_{k}),\alpha(\tau^n_{k}))b(X(u),\alpha(u))du ds 
\\
R_n(5)&:=\int_{t_n}^{t_{n+1}} \sum_{k =0}^{\nu(t_n,s)}\sum_{l=1}^m\int_{\tau^n_{k}}^{\tau^n_{k+1}}\mathcal{D} b(X(\tau^n_{k}),\alpha(\tau^n_{k})) \sigma_{(l)}(X(u),\alpha(u))dW_{l}(u) ds 
\\
R_n(6)&:= \sum_{l=1}^m  \mathbbm{1}_{\{N_n=1 \}} \Big( \sigma_{(l)}(X(\tau_1^n),\alpha_{n+1})-\sigma_{(l)}(X_n,\alpha_{n+1})
\\
&\qquad - \sigma_{(l)}(X(\tau_1^n),\alpha_{n})+\sigma_{(l)}(X_n,\alpha_{n}) \Big) \Big(W_l(t_{n+1})-W_l(\tau_1^n)\Big) 
\\
R_n(7)&:=\sum_{l=1}^m \mathbbm{1}_{\{N_n\geq 2\}} \int_{t_n}^{t_{n+1}} \sum_{i_0\neq j_0} \int_{t_n}^s\Big(\sigma_{(l)}(X(u),j_0)-\sigma_{(l)}(X(u),i_0)\Big)d[M_{i_0j_0}](u)dW_l(s)
\\
R_n(8)&:= \sum_{l=1}^m \int_{t_n}^{t_{n+1}} \sum_{i_0\neq j_0} \int_{t_n}^s\Big(\sigma_{(l)}(X(u),i_0)-\sigma_{(l)}(X(u),j_0)\Big)d\langle M_{i_0j_0}\rangle(u)dW_l(s)   
\\
R_n(9)&:=\sum_{l=1}^m\int_{t_n}^{t_{n+1}}  \sum_{j_0\in\mathcal{S}}\int_{t_n}^s q_{\alpha(u-)j_0}\Big(\sigma_{(l)}(X(u),j_0)-\sigma_{(l)} (X(u),\alpha(u-))\Big)du  dW_l(s) 
\\
R_n(10)&:=  \sum_{l=1}^m\int_{t_n}^{t_{n+1}}  \sum_{k =0}^{\nu(t_n,s)}\Big( \sigma_{(l)}(X(\tau^n_{k+1}),\alpha(\tau^n_{k}))-\sigma_{(l)}(X(\tau^n_{k}),\alpha(\tau^n_{k})) 
\\
&\qquad -\mathcal{D} \sigma_{(l)}(X(\tau^n_{k}),\alpha(\tau^n_{k}))( X(\tau^n_{k+1})-X(\tau^n_{k})) \Big) dW_l(s)
\\
R_n(11)&:=  \sum_{l=1}^m\int_{t_n}^{t_{n+1}}  \sum_{k =0}^{\nu(t_n,s)}\int_{\tau^n_{k}}^{\tau^n_{k+1}} \mathcal{D} \sigma_{(l)}(X(\tau^n_{k}),\alpha(\tau^n_{k}))b(X(u),\alpha(u))du dW_l(s)  
\\
R_n(12)&:= \sum_{l=1}^m\int_{t_n}^{t_{n+1}}  \sum_{k =0}^{\nu(t_n,s)} \sum_{l_1=1}^m\int_{\tau^n_{k}}^{\tau^n_{k+1}} \Big(\mathcal{D} \sigma_{(l)}(X(\tau^n_{k}),\alpha(\tau^n_{k})) \sigma_{(l_1)}(X(u),\alpha(u)) 
\\
& - \mathcal{D} \sigma_{(l)}(X_n,\alpha_n) \sigma_{(l_1)}(X_n,\alpha_n) \Big) dW_{l_1}(u) dW_l(s)  
\end{align*}
almost surely for any $n=0,1,\ldots, n_T-1$.  
\section{\bf{Moment Bound.}}
The proof of the following lemma can be found in \cite{mao2006} \textit{e.g.}, Theorems [3.3.16, 3.3.23, 3.3.24]. 
\begin{lemma}\label{l true moment and 1/2 rate}
Let Assumptions H-1 and H-2 be satisfied. Then, there exists a unique continuous solution $\{X(t); \{t\in [0,T]\}\}$ of SDEwMS \eqref{eq:sdems}. Moreover, the following hold, 
\begin{align*}
E\bigg(\sup_{t\in [0,T]}|X(t)|^{p}\Big|\mathscr{F}_T^{\alpha}\bigg)& \leq C
\\
E\bigg( \sup_{t\in [s,s+h]}|X(t)-X(s)|^p\Big|\mathscr{F}_T^{\alpha}  \bigg)& \leq Ch^{p/2} 
\end{align*}
where the positive constant $C$ is independent of $h$. 
\end{lemma}
\begin{lemma}
Let Assumptions H-1, H-2 and H-3 be satisfied. Then, 
\begin{align*}
E\Big(\sup_{n\in \{0,1,\ldots, n_T\}}|Y_n|^2\Big) \leq C
\end{align*}
where the positive constant $C$ does not depend on $h$. 
\end{lemma}
\begin{proof}
First notice that the Milstein scheme \eqref{eq:scheme} can be written as, 
\begin{align} 
Y_{n} & =Y_0+\sum_{k=0}^{n-1} b(Y_k, \alpha_k) h +\sum_{k=0}^{n-1} \sum_{l=1}^m\sigma_{(l)}(Y_k,\alpha_k)\Delta_k W_l\notag
\\
&+\sum_{k=0}^{n-1} \sum_{l,l_1=1}^m\int_{t_k}^{t_{k+1}}  \int_{t_k}^{s} \mathcal{D} \sigma_{(l)}(Y_k,\alpha_k) \sigma_{(l_1)}(Y_k,\alpha_k) dW_{l_1}(u) dW_l(s) \notag 
\\
&+\sum_{k=0}^{n-1} \sum_{l=1}^m \mathbbm{1}_{\{N_k=1\}}\Big(\sigma_{(l)}(Y_k,\alpha_{k+1})-\sigma_{(l)}(Y_k,\alpha_k)\Big)\Big(W_l(t_{k+1})-W_l(\tau_1^k)\Big) \notag
\end{align}
and hence one can obtain the following estimates,  
\begin{align} 
E\Big(&\sup_{n\in\{1,\ldots,n'\}} |Y_{n}|^2\Big)  \leq CE|Y_0|^2+C E\Big(\sup_{n\in\{1,\ldots,n'\}} \Big|\sum_{k=0}^{n-1} b(Y_k, \alpha_k) h\Big|^2\Big) \notag
\\
& +C E\Big(\sup_{n\in\{1,\ldots,n'\}} \Big| \sum_{k=0}^{n-1} \sum_{l=1}^m\sigma_{(l)}(Y_k,\alpha_k)\Delta_k W_l\Big|^2\Big)\notag
\\
&+C E\Big(\sup_{n\in\{1,\ldots,n'\}} \Big| \sum_{k=0}^{n-1} \sum_{l,l_1=1}^m\int_{t_k}^{t_{k+1}}  \int_{t_k}^{s} \mathcal{D} \sigma_{(l)}(Y_k,\alpha_k) \sigma_{(l_1)}(Y_k,\alpha_k) dW_{l_1}(u) dW_l(s) \Big|^2\Big) \notag 
\\
&+C E\Big(\sup_{n\in\{1,\ldots,n'\}} \Big| \sum_{k=0}^{n-1} \sum_{l=1}^m \mathbbm{1}_{\{N_k=1\}}\Big(\sigma_{(l)}(Y_k,\alpha_{k+1})-\sigma_{(l)}(Y_k,\alpha_k)\Big)\notag
\\
&\times \Big(W_l(t_{k+1})-W_l(\tau_1^k)\Big)\Big|^2\Big)  \notag
\\
& =: C E|Y_0|^p + T_1+ T_2+T_3+T_4 \label{eq:T1+T4} 
\end{align}
for any $n'=1,\ldots,n_T$.  For $T_1$, one can use Remark \ref{rem:growth} to obtain, 
\begin{align}
T_1 & := C E\Big(\sup_{n\in\{1,\ldots,n'\}} \Big|\sum_{k=0}^{n-1} b(Y_k, \alpha_k) h\Big|^2\Big)\leq C n'h^2 E\Big( \sum_{k=0}^{n'-1} |b(Y_k, \alpha_k)|^2 \Big) \notag
\\
& \leq C + C h  \sum_{k=0}^{n'-1} E\Big(\sup_{n\in\{0,\ldots,k\}}|Y_n|^2 \Big) \label{eq:T1}
\end{align}
for any $n'=1,\ldots,n_T$. Notice that 
$$
\Big\{\sum_{k=0}^{n-1} \sum_{l=1}^m\sigma_{(l)}(Y_k,\alpha_k)\Delta_k W_l; n\in\{1,\ldots,n_T\}\Big\}
$$
is a  martingale with respect to filtration $\{\mathscr{F}_T^\alpha\vee\mathscr{F}_{t_n}^W;n\in\{1,\ldots,n_T\}\}$. Hence, by Burkholder-Davis-Gundy inequality and Remark \ref{rem:growth}, $T_2$ can be estimates as follows, 
\begin{align}
T_2 & := C E\Big(\sup_{n\in\{1,\ldots,n' \}} \Big| \sum_{k=0}^{n-1} \sum_{l=1}^m\sigma_{(l)}(Y_k,\alpha_k)\Delta_k W_l \Big|^2 \Big)\notag
\\
& \leq C h E\Big(  \sum_{k=0}^{n'-1} \sum_{l=1}^m|\sigma_{(l)}(Y_k,\alpha_k)|^2 \Big) \notag 
\\
& \leq C + C h  \sum_{k=0}^{n'-1} E\Big(\sup_{n\in\{0,\ldots,k\}}|Y_n|^2 \Big) \label{eq:T2}
\end{align}
for any $n'=1,\ldots,n_T$. Further, one can show that 
$$
 \Big\{\sum_{k=0}^{n-1} \sum_{l,l_1=1}^m\int_{t_k}^{t_{k+1}}  \int_{t_k}^{s} \mathcal{D} \sigma_{(l)}(Y_k,\alpha_k) \sigma_{(l_1)}(Y_k,\alpha_k) dW_{l_1}(u) dW_l(s); n\in\{1,\ldots n_T\} \Big\}
$$
is a  martingale with respect to filtration $\{\mathscr{F}_T^\alpha\vee\mathscr{F}_{t_n}^W;n\in\{1,\ldots,n_T\}\}$. Thus, as before, due to Burkholder-Davis-Gundy inequality and Remark \ref{rem:growth}, one obtains 
\begin{align}
T_3 & := C E\Big(\sup_{n\in\{1,\ldots,n'\}} \Big| \sum_{k=0}^{n-1} \sum_{l,l_1=1}^m\int_{t_k}^{t_{k+1}}  \int_{t_k}^{s} \mathcal{D} \sigma_{(l)}(Y_k,\alpha_k) \sigma_{(l_1)}(Y_k,\alpha_k) dW_{l_1}(u) dW_l(s) \Big|^2\Big) \notag 
\\
& \leq C E\Big(\sum_{k=0}^{n'-1} \sum_{l,l_1=1}^m\int_{t_k}^{t_{k+1}}  \int_{t_k}^{s} |\mathcal{D} \sigma_{(l)}(Y_k,\alpha_k) \sigma_{(l_1)}(Y_k,\alpha_k)|^2  du ds \Big) \notag
\\
& \leq C + C h  \sum_{k=0}^{n'-1} E\Big(\sup_{n\in\{0,\ldots,k\}}|Y_n|^2 \Big) \label{eq:T3}
\end{align}
for any $n'=1,\ldots,n_T$. Similarly, one can show that 
\begin{align*}
\Big\{ \sum_{k=0}^{n-1} \sum_{l=1}^m \mathbbm{1}_{\{N_k=1\}}\Big(\sigma_{(l)}(Y_k,\alpha_{k+1})&-\sigma_{(l)}(Y_k,\alpha_k)\Big)\Big(W_l(t_{k+1})-W_l(\tau_1^k)\Big); 
\\
&  n\in\{1,\ldots n_T\}\Big\}
\end{align*}
is a  martingale with respect to filtration $\{\mathscr{F}_T^\alpha\vee\mathscr{F}_{t_n}^W;n\in\{1,\ldots,n_T\}\}$.  Again, on the application of Burkholder-Davis-Gundy inequality and Remark \ref{rem:growth}, one obtains
\begin{align}
T_4& := C E\Big(\sup_{n\in\{1,\ldots,n'\}} \Big| \sum_{k=0}^{n-1} \sum_{l=1}^m \mathbbm{1}_{\{N_k=1\}}\Big(\sigma_{(l)}(Y_k,\alpha_{k+1})-\sigma_{(l)}(Y_k,\alpha_k)\Big)\notag
\\
& \qquad \times \Big(W_l(t_{k+1})-W_l(\tau_1^k)\Big)\Big|^2\Big) \notag
\\
& \leq C E\Big(  \sum_{k=0}^{n'-1} \sum_{l=1}^m \mathbbm{1}_{\{N_k=1\}}\big|\sigma_{(l)}(Y_k,\alpha_{k+1})-\sigma_{(l)}(Y_k,\alpha_k)\big|^2 \notag
\\
& \qquad \times E\big(|W_l(t_{k+1})-W_l(\tau_1^k)\big|^2\big|\mathscr{F}_T^\alpha \vee \mathscr{F}_{t_k}^W \big)\Big) \notag 
\\
& \leq C + C h  \sum_{k=0}^{n'-1} E\Big(\sup_{n\in\{0,\ldots,k\}}|Y_n|^2 \Big) \label{eq:T4}
\end{align}
for any $n'=1,\ldots,n_T$. Substituting the values from \eqref{eq:T1} to \eqref{eq:T4} in \eqref{eq:T1+T4} gives, 
\begin{align}
E\Big(&\sup_{n\in\{0,\ldots,n'\}} |Y_{n}|^2\Big)  \leq CE|Y_0|^2 + C + C h  \sum_{k=0}^{n'-1} E\Big(\sup_{n\in\{0,\ldots,k\}}|Y_n|^2 \Big)  \notag
\end{align}
for any $n'=1,\ldots,n_T$. The application of Gronwall's lemma completes the proof. 
\end{proof}

\section{\bf{Proof of Main Result}.}
\noindent
Before proving the main result stated in Theorem \ref{thm:main}, one requires to establish several lemmas  which now follows.  

\begin{lemma}\label{l mvt}
Let $f(\cdot,i_0):\mathbb{R}^d \rightarrow \mathbb{R}^d$ be a continuously differentiable function and satisfies, for every $i_0\in\mathcal{S}$,  
\begin{align} \label{eq:hyp}
|\mathcal{D}f(x,i_0)-\mathcal{D}f(\tilde{x},i_0)|\leq C |x-\tilde{x}| 
\end{align}
for any $x,\tilde{x}\in \mathbb{R}^d$. Then, for every $i_0\in\mathcal{S}$,
\begin{equation*}
|f(x,i_0)-f(\tilde{x},i_0)-\mathcal{D} f(\tilde{x},i_0)(x-\tilde{x})|\leq C|x-\tilde{x}|^2
\end{equation*}
for any $x,\tilde{x}\in \mathbb{R}^d$. In the above, $C>0$ is  constant. 
\end{lemma}
\begin{proof}
For every $i_0\in\mathcal{S}$, due to mean value theorem, 
\begin{align*}
f(x,i_0)-f(\tilde{x}.i_0)=\mathcal{D} f(qx+(1-q)\tilde{x},i_0)(x-\tilde{x})
\end{align*}
for some $q\in (0,1)$ which on using  hypothesis  \eqref{eq:hyp}  further implies, 
\begin{align*}
|f(x,i_0)&-f(\tilde{x},i_0)-\mathcal{D} f(\tilde{x},i_0)(x-\tilde{x})|
\\
=&\Big|\mathcal{D} f(qx+(1-q)\tilde{x},i_0)(x-\tilde{x})- \mathcal{D} f(\tilde{x},i_0)(x-\tilde{x}) \Big|
\\
\leq & C|qx+(1-q)\tilde{x}-\tilde{x}||x-\tilde{x}|\leq |x-\tilde{x}|^2
\end{align*}
for any $x,\tilde{x}\in \mathbb{R}^d$. This completes the proof. 
\end{proof}
The proof of parts (a) and (b) of the following lemma can be found in \cite{nguyen2017}. For the completeness, their proofs are given below along with that of part (c).  
\begin{lemma}\label{l EN_i, EN_i^2 }
Let $q:=\max \{-q_{i_0i_0}; i_0 \in \mathcal{S}\}$.  
\newline
(a). For any $n=0,1,\dots,n_T-1$, one has $P(N_n\geq N)\leq q^Nh^N$ whenever $N\geq 1$.
\newline
(b). If $h<1/(2q)$, then $EN_n \leq C h$ for any $n=0,1,\ldots,n_T-1$	 where $C>0$ is a constant independent of $h$. 
\newline
(c). Also,   $ EN_n^2 \leq 6$ for any $n=0,1,\ldots,n_T-1$. 
\end{lemma}
\begin{proof}
Recall that $\tau_1^n,\ldots,\tau_{N_n}^n$ are jump-times of the chain $\alpha$ in the interval $(t_n,t_{n+1}]$ and $\tau_0^n=t_n$, $\tau_{N_n+1}^n=t_{n+1}$. Clearly, inter-jump times $\tau_1^n-\tau_0^n$, $\tau_2^n-\tau_1^n$, $\tau_3^n-\tau_2^n$, $\ldots$, $\tau_{N_n-1}^n-\tau_{N_n}^n$, $\tau_{N_n+1}^n-\tau_{N_n}^n$ are conditionally independent random variables on $\{N_n\geq 1\}$. Further, if $N_n\geq 1$ and at time $\tau_r^n$,  chain jumps from state $i_{r-1}$ to $i_r$ for $r=1,\ldots,N_n$, then the random variable $\tau_{r+1}^n-\tau_r^n$ follows  exponential distribution with parameter $-q_{i_ri_r}$. Hence, by strong Markov property of $\alpha$, 
\begin{align*}
P(N_n\geq N)& \leq P\Big(\sum_{r=0}^{N-1} (\tau_{r+1}^n-\tau_r^n)<h\Big)\leq \prod_{r=0}^{N-1}P(\tau_{r+1}^n-\tau_r^n<h) 
\\
& \leq \prod_{r=0}^{N-1}(1-e^{q_{i_r i_r}h}) \leq \prod_{r=0}^{N-1} (-q_{i_r i_r}h) \leq q^Nh^N 
\end{align*}
for any $N\geq 1$ and for any $n=0,1,\ldots,n_T-1$, which shows part (a). For part (b), one writes, 
\begin{align*}
EN_n=\sum_{N=1}^\infty P(N_n\geq N) \leq \sum_{N=1}^\infty q^Nh^N \leq q h  \sum_{N=0}^\infty (1/2)^N \leq C h 
\end{align*}    
for any $n=0,1,\ldots,n_T-1$.	Furthermore, 
\begin{align*}
EN_n^2=\sum_{N=1}^\infty & N^2 P(N_n=N)\leq \sum_{N=1}^\infty N^2 P(N_n\geq N) \leq  \sum_{N=1}^\infty N^2 q^N h^N 
\\
&\leq \sum_{N=1}^\infty N^2 (1/2)^N =6
\end{align*}
for any $n=0,1,\ldots,n_T-1$, which proves part (c). 
\end{proof}
 
\begin{lemma} \label{lem:R3R10}
Let Assumptions H-1, H-2 and H-3 be satisfied. Then, there exists a positive constant $C$ such that,  
\begin{align*}
E\Big(\sup_{n'\in\{1,\ldots,n_T\}}\Big|\sum_{n=0}^{n'-1}R_{n}(3)\Big|^2\Big) \leq C h^2, \mbox{ and } E\Big(\sup_{n'\in\{1,\ldots,n_T\}}\sum_{n=0}^{n'-1}|R_{n}(10)|^2\Big) \leq C h^2
\end{align*}
where constant $C>0$ does not depend on $h$. 
\end{lemma}
\begin{proof}
For the first term, one applies H\"older's inequality to obtain the following, 
\begin{align*}
& E\Big(\sup_{n'\in\{1,\ldots,n_T\}}\Big|\sum_{n=0}^{n'-1}R_{n}(3)\Big|^2\Big)
\\
& =E\Big(\sup_{n'\in\{1,\ldots,n_T\}} \Big|\sum_{n=0}^{n'-1} \int_{t_n}^{t_{n+1}}\sum_{k =0}^{\nu(t_n,s)}\Big( b(X(\tau_{k+1}^n),\alpha(\tau_{k}^n))-b(X(\tau_{k}^n),\alpha(\tau_{k}^n)) \notag
\\
&\quad -\mathcal{D} b(X(\tau^n_{k}),\alpha(\tau_{k}))( X(\tau_{k+1}^n)-X(\tau_{k}^n))\Big) ds\Big|^2 \Big)
\\
&\leq Chn_T E\sum_{n=0}^{n_T-1} \int_{t_n}^{t_{n+1}} (1+\nu(t_n,s)) \sum_{k =0}^{\nu(t_n,s)}E\Big(\big| b(X(\tau_{k+1}^n),\alpha(\tau_{k}^n))-b(X(\tau_{k}^n),\alpha(\tau_{k}^n)) \notag  
\\
&\quad -\mathcal{D} b(X(\tau^n_{k}),\alpha(\tau_{k}))( X(\tau_{k+1}^n)-X(\tau_{k}^n)) \big|^2\big|\mathscr{F}_T^\alpha\Big) ds
\end{align*}
which on using Lemma \ref{l true moment and 1/2 rate}, Lemma \ref{l mvt} and Lemma \ref{l EN_i, EN_i^2 } along with $\nu(t_n,s)\leq N_n$ yields the following estimate, 
\begin{align*}
& E\Big(\sup_{n'\in\{1,\ldots,n_T\}}\Big|\sum_{n=0}^{n'-1}R_{n}(3)\Big|^2\Big) 
\\
&  \leq C \sum_{n=0}^{n_T-1} \int_{t_n}^{t_{n+1}}E\Big((1+\nu(t_n,s))\sum_{k =0}^{\nu(t_n,s)}E\Big(|X(\tau_{k+1}^n)-X(\tau_{k}^n)|^4\Big|\mathcal{F}_T^{\alpha}\Big)\Big)ds
\\
& \leq C\sum_{n=0}^{n_T-1} \int_{t_n}^{t_{n+1}}E\Big((1+\nu(t_n,s))\sum_{k =0}^{\nu(t_n,s)}|\tau_{k+1}^n-\tau_{k}^n)|^2\Big)ds
\\
&\leq  Ch^2\sum_{n=0}^{n_T-1}  \int_{t_n}^{t_{n+1}}E\big((1+\nu(t_n,s)\big)^2 ds\leq Ch^2 \sum_{n=0}^{n_T-1}  \int_{t_n}^{t_{n+1}}(1+E(N_n^2))ds\leq Ch^2.
\end{align*}
 For the second term, notice that $\{\sum_{n=0}^{n'-1}R_n(10);n'\in\{1,\ldots,n_T\}\}$ is a square integrable martingale with respect to filtration $\{\mathscr{F}^\alpha_T\vee \mathscr{F}_{t_{n'}}^W;n'\in\{1,\ldots,n_T\}\}$. Due to Burkholder-Davis-Gundy inequality and H\"older's inequality, one obtains
\begin{align*}
& E\Big(\sup_{n'\in\{1,\ldots,n_T\}}\Big|\sum_{n=0}^{n'-1}R_{n}(10)\Big|^2\Big)
\\
& = E\Big(\sup_{n'\in\{1,\ldots,n_T\}}\Big|  \sum_{n=0}^{n'-1} \sum_{l=1}^m\int_{t_n}^{t_{n+1}}\sum_{k =0}^{\nu(t_n,s)}\Big( \sigma_{(l)}(X(\tau_{k+1}^n),\alpha(\tau_{k}^n))  \notag
\\
&-\sigma_{(l)}(X(\tau_{k}^n),\alpha(\tau_{k}^n)) -\mathcal{D} \sigma_{(l)}(X(\tau_{k}^n),\alpha(\tau_{k}^n))( X(\tau_{k+1}^n)-X(\tau_{k}^n)) \Big)dW_l(s)  \Big|^2 \Big)
\\
& \leq C E \sum_{n=0}^{n_T-1} \sum_{l=1}^m\int_{t_n}^{t_{n+1}}(1+\nu(t_n,s)) \sum_{k =0}^{\nu(t_n,s)} E\Big(\big| \sigma_{(l)}(X(\tau_{k+1}^n),\alpha(\tau_{k}^n))  \notag
\\
&-\sigma_{(l)}(X(\tau_{k}^n),\alpha(\tau_{k}^n)) -\mathcal{D} \sigma_{(l)}(X(\tau_{k}^n),\alpha(\tau_{k}^n))( X(\tau_{k+1}^n)-X(\tau_{k}^n)) \big|^2 \Big|\mathscr{F}_T^\alpha\Big) ds 
\end{align*}
which on using Lemma \ref{l true moment and 1/2 rate}, Lemma \ref{l mvt} and Lemma \ref{l EN_i, EN_i^2 } along with $\nu(t_n,s)\leq N_n$ yields the following estimate,
\begin{align*}
E\Big(&\sup_{n'\in\{1,\ldots,n_T\}}\Big|\sum_{n=0}^{n'-1}R_{n}(10)\Big|^2\Big)
\\
 & \leq C  E  \Big( \sum_{n=0}^{n_T-1}  \int_{t_n}^{t_{n+1}} (1+\nu(t_n,s)) \sum_{k =0}^{\nu(t_n,s)} E\Big(| X(\tau_{k+1}^n)-X(\tau_{k}^n) |^4\Big|\mathscr{F}_T^\alpha\Big) ds 
 \\
 & \leq C   h^2. 
\end{align*}
Thus, the proof is completed. 
\end{proof}
\begin{lemma} \label{lem:R4R5R11}
Let Assumptions H-1, H-2 and H-3 be satisfied. Then, 
\begin{align*}
E\Big(\sup_{n'\in\{1,\ldots,n_T\}} \Big| \sum_{n=0}^{n'-1} R_{n}(4)\Big|^2 \Big) & \leq C h^2, 
\\
E\Big(\sup_{n'\in\{1,\ldots,n_T\}} \Big| \sum_{n=0}^{n'-1}R_{n}(5)\Big|^2\Big) & \leq C h^2,
\\
\,\, \mbox{and} \,\,  E\Big(\sup_{n'\in\{1,\ldots,n_T\}} \Big| \sum_{n=0}^{n'-1}R_{n}(11)\Big|^2\Big) & \leq C h^2 
\end{align*}
where the constant $C>0$ does not depend on $h$. 
\end{lemma}
\begin{proof} By using H\"older's inequality, 
\begin{align*}
& E\Big(\sup_{n'\in\{1,\ldots,n_T\}}\Big|\sum_{n=0}^{n'-1}R_{n}(4)\Big|^2 \Big)  
\\
&=E\Big(\sup_{n'\in\{1,\ldots,n_T\}}\Big|\sum_{n=0}^{n'-1} \int_{t_n}^{t_{n+1}} \sum_{k =0}^{\nu(t_n,s)} \int_{\tau^n_{k}}^{\tau^n_{k+1}}\mathcal{D} b(X(\tau^n_{k}),\alpha(\tau^n_{k}))b(X(u),\alpha(u))du ds \Big|^2\Big)
\\
& \leq C h^2n_T E\sum_{n=0}^{n_T-1}\int_{t_n}^{t_{n+1}} (1+\nu(t_n,s)) 
\\
& \qquad \times \sum_{k =0}^{\nu(t_n,s)} \int_{\tau^n_{k}}^{\tau^n_{k+1}}E\Big(|\mathcal{D} b(X(\tau^n_{k}),\alpha(\tau^n_{k}))|^2 |b(X(u),\alpha(u))|^2\Big|\mathscr{F}_{T}^\alpha\Big) du ds
\end{align*}
which on using Remark \ref{rem:growth}, Lemma \ref{l true moment and 1/2 rate} and Lemma \ref{l EN_i, EN_i^2 } give, 
\begin{align*}
& E\Big(\sup_{n'\in\{1,\ldots,n_T\}}\Big|\sum_{n=0}^{n'-1}R_{n}(4)\Big|^2 \Big)    
\\
& \leq C h E\sum_{n=0}^{n_T-1} \int_{t_n}^{t_{n+1}} (1+\nu(t_n,s)) \sum_{k =0}^{\nu(t_n,s)} \int_{\tau^n_{k}}^{\tau^n_{k+1}} E\Big(\big(1+|X(u)|^2\big)|\mathscr{F}^\alpha_T\Big) du ds
\\
& \leq C h E\sum_{n=0}^{n_T-1} \int_{t_n}^{t_{n+1}} (1+\nu(t_n,s)) (s-t_n)  ds \leq C h^2.
\end{align*}
Notice that $\{\sum_{n=0}^{n'-1}R_n(5);n'\in\{1,\ldots,n_T\}\}$ is a square integrable martingale with respect to filtration $\{\mathscr{F}_T^\alpha\vee\mathscr{F}_{t_{n'}}^W;n'\in\{0,\ldots,n_T\}\}$.  Due to Burkholder-Davis-Gundy inequality,  H\"older's inequality and Remark \ref{rem:growth}, one obtains
\begin{align*}
& E\Big(\sup_{n'\in\{1,\ldots,n_T\}}\Big|\sum_{n=0}^{n'-1}R_n(5)\Big|^2\Big)
\\
&=E\Big(\sup_{n'\in\{1,\ldots,n_T\}}\Big|\sum_{n=0}^{n'-1} \int_{t_n}^{t_{n+1}} \sum_{k =0}^{\nu(t_n,s)}
\\
& \qquad \sum_{l=1}^m\int_{\tau^n_{k}}^{\tau^n_{k+1}}\mathcal{D} b(X(\tau^n_{k}),\alpha(\tau^n_{k})) \sigma_{(l)}(X(u),\alpha(u))dW_{l}(u) ds\Big|^2\Big)
\\
&\leq C h E\sum_{n=0}^{n_T-1} \int_{t_n}^{t_{n+1}} (1+\nu(t_n,s))
\\
& \qquad \times \sum_{k =0}^{\nu(t_n,s)}\int_{\tau^n_{k}}^{\tau^n_{k+1}} E\Big(|\mathcal{D} b(X(\tau^n_{k}),\alpha(\tau^n_{k}))|^2 | \sigma_{(l)}(X(u),\alpha(u))|^2 \Big|\mathscr{F}_T^\alpha\Big) du ds
\\
&\leq C  h E\sum_{n=0}^{n_T-1} \int_{t_n}^{t_{n+1}} (1+\nu(t_n,s))\sum_{k =0}^{\nu(t_n,s)}\int_{\tau^n_{k}}^{\tau^n_{k+1}} E\Big((1+ |X(u)|^2) \big|\mathscr{F}_T^\alpha\Big) du ds
\end{align*}
which on the application of Lemma \ref{l true moment and 1/2 rate} and Lemma \ref{l EN_i, EN_i^2 } gives, 
\begin{align*}
E\Big(\sup_{n'\in\{1,\ldots,n_T\}}\Big|\sum_{n=0}^{n'-1}R_n(5)\Big|^2\Big)& \leq C h E\sum_{n=0}^{n_T-1} \int_{t_n}^{t_{n+1}} (1+\nu(t_n,s))(s-t_n) ds \leq Ch^2.
\end{align*}
Further, notice that $\{\sum_{n=0}^{n'-1} R_n(11), n'=\{1,\ldots,n_T\}\}$ is a square integrable martingale with respect to filtration $\{\mathscr{F}_T^\alpha\vee\mathscr{F}_{t_{n'}}^W; n'\in\{1,\ldots,n_T\}\}$. By Burkholder-Davis-Gundy inequality and H\"older's inequality, one obtains 
\begin{align*}
& E\Big(\sup_{n'\in\{1,\ldots,n_T\}}\Big|\sum_{n=0}^{n'-1}R_n(11)\Big|^2\Big)
\\
&= E\Big(\sup_{n'\in\{1,\ldots,n_T\}}\Big| \sum_{n=0}^{n'-1}\sum_{l=1}^m\int_{t_n}^{t_{n+1}}  \sum_{k =0}^{\nu(t_n,s)}\int_{\tau^n_{k}}^{\tau^n_{k+1}} \mathcal{D} \sigma_{(l)}(X(\tau^n_{k}),\alpha(\tau^n_{k}))
\\
&\qquad \times b(X(u),\alpha(u))du dW_l(s)\Big|^2 \Big)
\\
&\leq C h E \sum_{n=0}^{n_T-1}\sum_{l=1}^m \int_{t_n}^{t_{n+1}} (\nu(t_n,s)+1) \sum_{k =0}^{\nu(t_n,s)}\int_{\tau^n_{k}}^{\tau^n_{k+1}} E\big( |\mathcal{D} \sigma_{(l)}(X(\tau^n_{k}),\alpha(\tau^n_{k}))
\\
&\qquad \times b(X(u),\alpha(u))|^2 | \mathscr{F}_T^\alpha\big) du ds
\end{align*}
which due to Remark \ref{rem:growth}, Lemma \ref{l true moment and 1/2 rate} and Lemma \ref{l EN_i, EN_i^2 }, yields the following estimate, 
\begin{align*}
E\Big(&\sup_{n'\in\{1,\ldots,n_T\}}\Big|\sum_{n=0}^{n'-1}R_n(11)\Big|^2\Big)
\\
& \leq C h E \sum_{n=0}^{n_T-1}\int_{t_n}^{t_{n+1}} (\nu(t_n,s)+1) 
\\
& \qquad \sum_{k =0}^{\nu(t_n,s)}\int_{\tau^n_{k}}^{\tau^n_{k+1}}  \Big(1+E\big(|X(u)|^2|\mathscr{F}_T^\alpha\big) \Big) du ds\leq C h^2. 
\end{align*}
This completes the proof of the lemma. 
\end{proof}
\begin{lemma} \label{lem:R2R9}
Let Assumptions H-1, H-2 and H-3 be satisfied. Then,   
\begin{align*}
E\Big(\sup_{n'\in\{1,\ldots,n_T\}}\Big|\sum_{n=0}^{n'-1}R_n(2)\Big|^2\Big) \leq Ch^{2}, \,\, E\Big(\sup_{n'\in\{1,\ldots,n_T\}}\Big|\sum_{n=0}^{n'-1}R_n(9)\Big|^2\Big) \leq Ch^{2}, 
\end{align*}
where constant $C>0$ does not depend on $h$. 
\begin{proof}
By using H\"older's inequality, Remark \ref{rem:growth},  and Lemma \ref{l true moment and 1/2 rate}, one obtains, 
\begin{align*}
& E\Big(\sup_{n'\in\{1,\ldots,n_T\}}\Big|\sum_{n=0}^{n'-1}R_n(2)\Big|^2\Big)
\\
&=E\Big(\sup_{n'\in\{1,\ldots,n_T\}}\Big|\sum_{n=0}^{n'-1}\int_{t_n}^{t_{n+1}} \sum_{j_0\in\mathcal{S}}\int_{t_n}^s q_{\alpha(u-)j_0}\Big(b(X(u),j_0)
\\
&\qquad -b(X(u),\alpha(u-))\Big)du ds\Big|^2 \Big)
\\
&\leq Ch^2 n_T E\sum_{n=0}^{n_T-1}\int_{t_n}^{t_{n+1}} \sum_{j_0\in\mathcal{S}}\int_{t_n}^s (q_{\alpha(u-)j_0})^2E\Big( \big|b(X(u),j_0)
\\
&\qquad-b(X(u),\alpha(u-))\big|^2\Big| \mathscr{F}_T^\alpha \Big)du ds  
\\
&\leq Ch E\sum_{n=0}^{n_T-1}\int_{t_n}^{t_{n+1}} \int_{t_n}^s E\Big( \big(1+|X(u)|^2 \big) \Big| \mathscr{F}_T^\alpha \Big)du ds  \leq Ch^2.
\end{align*}
Again, notice that $\{\sum_{n=0}^{n'-1}R_n(9);n'\in\{1,\ldots,n_T\}\}$ is a square integrable martingale with respect to the filtration $\{\mathscr{F}_T^\alpha\vee\mathscr{F}_{t_{n'}}^W;n'\in\{1,\ldots,n_T\}\}$.  As before, one uses Burkholder-Davis-Gundy inequality, H\"older's inequality, Remark \ref{rem:growth},  and Lemma \ref{l true moment and 1/2 rate} to obtain the following estimate, 
\begin{align*}
& E\Big(\sup_{n'\in\{1,\ldots,n_T\}}\Big|\sum_{n=0}^{n'-1}R_n(9)\Big|^2\Big) 
\\
& =E\Big(\sup_{n'\in\{1,\ldots,n_T\}} \Big|\sum_{n=0}^{n'-1}\sum_{l=1}^m\int_{t_n}^{t_{n+1}}  \sum_{j_0\in\mathcal{S}}\int_{t_n}^s q_{\alpha(u-)j_0}\Big(\sigma_{(l)}(X(u),j_0)
\\
&\qquad  -\sigma_{(l)} (X(u),\alpha(u-))\Big)du  dW_l(s)\Big|^2\Big) 
\\
&\leq C h E\sum_{n=0}^{n_T-1}\sum_{l=1}^m\int_{t_n}^{t_{n+1}}  \sum_{j_0\in\mathcal{S}}\int_{t_n}^s (q_{\alpha(u-)j_0})^2 E\Big(|\sigma_{(l)}(X(u),j_0)
\\
&\qquad-\sigma_{(l)} (X(u),\alpha(u-))|^2\Big|\mathscr{F}_T^\alpha\Big) du  ds
\\
&\leq Ch E\sum_{n=0}^{n_T-1}\int_{t_n}^{t_{n+1}} \int_{t_n}^s E\Big( \big(1+|X(u)|^2 \big) \Big| \mathscr{F}_T^\alpha \Big)du ds  \leq Ch^2
\end{align*}
 which completes the proof. 
\end{proof}
\end{lemma}

\begin{lemma} \label{lem:R1R6R7R8R12}
Let Assumptions H-1, H-2 and H-3 be satisfied. Then, 
\begin{align*}
E\Big(\sup_{n'\in\{1,\ldots,n_T\}}\Big|\sum_{n=0}^{n'-1}R_n(1)\Big|^2\Big)  & \leq Ch^{2}\,\, \mbox{and}\,\, E\Big(\sup_{n'\in\{1,\ldots,n_T\}}\Big|\sum_{n=0}^{n'-1}R_n(6)\Big|^2\Big)   \leq Ch^{2}
\\
E\Big(\sup_{n'\in\{1,\ldots,n_T\}}\Big|\sum_{n=0}^{n'-1}R_n(7)\Big|^2\Big) & \leq Ch^{2}  \,\, \mbox{and}\,\, E\Big(\sup_{n'\in\{1,\ldots,n_T\}}\Big|\sum_{n=0}^{n'-1}R_n(8)\Big|^2\Big)  \leq Ch^{2}
\\
E\Big(\sup_{n'\in\{1,\ldots,n_T\}}\Big|& \sum_{n=0}^{n'-1}R_n(12)\Big|^2\Big)  \leq Ch^{2}
\end{align*}
where the constant $C>0$ does not depend on $h$. 
\end{lemma}
\begin{proof}
First, observe that $\{\sum_{n=0}^{n'-1}R_n(1);n'\in\{1,\ldots,n_T\}\}$ is a square integrable martingale with respect to filtration $\{\mathscr{F}_{t_{n'}}; n'\in\{1,\ldots,n_T\}\}$. Hence, using Burkholder-Davis-Gundy inequality and H\"older's inequality, one obtains, 
\begin{align*}
& E\Big(\sup_{n'\in\{1,\ldots,n_T\}}\Big|\sum_{n=0}^{n'-1}R_n(1)\Big|^2\Big)
\\
&= E\Big(\sup_{n'\in\{1,\ldots,n_T\}}\Big|\sum_{n=0}^{n'-1}\int_{t_n}^{t_{n+1}} \sum_{i_0\neq j_0}\int_{t_n}^s\Big(b(X(u),j_0)-b(X(u),i_0)\Big)dM_{i_0j_0}(u) ds\Big|^2 \Big)
\\
& \leq C h E\sum_{n=0}^{n_T-1} \int_{t_n}^{t_{n+1}} \sum_{i_0\neq j_0}\Big| \int_{t_n}^s\Big(b(X(u),j_0)-b(X(u),i_0)\Big)dM_{i_0j_0}(u)\Big|^2  ds
\\
& \leq C h E\sum_{n=0}^{n_T-1}  \int_{t_n}^{t_{n+1}} \sum_{i_0\neq j_0} \int_{t_n}^s |b(X(u),j_0)-b(X(u),i_0)|^2 d[M_{i_0j_0}](u)  ds
\end{align*}
which due to Remark \ref{rem:growth}, Lemma \ref{l true moment and 1/2 rate} and Lemma \ref{l EN_i, EN_i^2 } gives, 
\begin{align*}
E\Big(&\sup_{n'\in\{1,\ldots,n_T\}}\Big|\sum_{n=0}^{n'-1}R_n(1)\Big|^2\Big) 
\\
& \leq C h E \sum_{n=0}^{n_T-1} \int_{t_n}^{t_{n+1}} \sum_{i_0\neq j_0} \int_{t_n}^s \Big(1+ \sup_{0\leq t\leq T}|X(u)|^2\Big) d[M_{i_0j_0}](u)  ds
\\
& \leq C h E\Big( \Big(1+ E\Big(\sup_{0\leq t\leq T}|X(u)|^2\Big|\mathscr{F}_T^\alpha\Big) \Big)
\\
&\times  \sum_{n=0}^{n_T-1} \int_{t_n}^{t_{n+1}} \sum_{i_0\neq j_0} \big([M_{i_0j_0}](s)-[M_{i_0j_0}](t_n) \big)  ds\Big)
\\
& \leq C h^2 \sum_{n=0}^{n_T-1} E (N_n)  \leq C h^2. 
\end{align*}
Again, notice that $\{\sum_{n=0}^{n'-1}R_n(6);n'\in\{1,\ldots,n_T\}\}$ is a square integrable martingale with respect to filtration $\{\mathscr{F}_T^\alpha\vee \mathscr{F}_{t_{n'}};n'\in\{1,\ldots,n_T\}\}$. Thus, due to Burkholder-Davis-Gundy inequality and H\"older's inequality, one obtains,  
\begin{align*}
E\Big(&\sup_{n'\in\{1,\ldots,n_T\}}\Big|\sum_{n=0}^{n'-1}R_n(6)\Big|^2\Big)
\\
&= E\Big(\sup_{n'\in\{1,\ldots,n_T\}}\Big|\sum_{n=0}^{n'-1} \sum_{l=1}^m  \mathbbm{1}_{\{N_n=1 \}} \Big( \sigma_{(l)}(X(\tau_1^n),\alpha_{n+1})-\sigma_{(l)}(X_n,\alpha_{n+1})
\\
&\qquad -\sigma_{(l)}(X(\tau_1^n),\alpha_{n})+\sigma_{(l)}(X_n,\alpha_{n})\Big) \Big(W_l(t_{n+1})-W_l(\tau_1^n)\Big) \Big|^2 \Big)
\\
& \leq  C E\sum_{n=0}^{n_T-1} \sum_{l=1}^m  \mathbbm{1}_{\{N_n=1 \}} \big(|\sigma_{(l)}(X(\tau_1^n),\alpha_{n+1})-\sigma_{(l)}(X_n,\alpha_{n+1})|^2
\\
& \qquad +|\sigma_{(l)}(X(\tau_1^n),\alpha_{n})-\sigma_{(l)}(X_n,\alpha_{n})|^2\big) |W_l(t_{n+1})-W_l(\tau_1^n)|^2  
\\
& =  C E\sum_{n=0}^{n_T-1} \sum_{l=1}^m  \mathbbm{1}_{\{N_n=1 \}} \big(|\sigma_{(l)}(X(\tau_1^n),\alpha_{n+1})-\sigma_{(l)}(X_n,\alpha_{n+1})|^2
\\
& \qquad + |\sigma_{(l)}(X(\tau_1^n),\alpha_{n})-\sigma_{(l)}(X_n,\alpha_{n})|^2 \big) 
\\
&\qquad\times E(|W_l(t_{n+1})-W_l(\tau_1^n)|^2|\mathscr{F}_{t_n} \vee \mathscr{F}_T^\alpha )
\end{align*}
which due to Assumption H-2, Lemma \ref{l true moment and 1/2 rate} and Lemma \ref{l EN_i, EN_i^2 } gives, 
\begin{align*}
E\Big(&\sup_{n'\in\{1,\ldots,n_T\}}\Big|\sum_{n=0}^{n'-1}R_n(6)\Big|^2\Big)
\\
& \leq   C h E  \sum_{n=0}^{n_T-1} \mathbbm{1}_{\{N_n=1 \}} E\big( |X(\tau_1^n)-X_n|^2 |\mathscr{F}_T^\alpha\big) \leq   C h^2   \sum_{n=0}^{n_T-1} P(N_n\geq 1 )  \leq C h^2. 
\end{align*}
Furthermore, it is clear that $\{\sum_{n=0}^{n'-1}R_n(7);n'\in\{1,\ldots,n_T\}\}$ is a square integrable martingale with respect to filtration $\{\mathscr{F}_T^\alpha\vee\mathscr{F}_{t_{n'}};n'\in\{1,\ldots,n_T\}\}$. So, one uses Burkholder-Davis-Gundy inequality and  H\"older's inequality to obtain,  
\begin{align*}
E\Big(&\sup_{n'\in\{1,\ldots,n_T\}}\Big|\sum_{n=0}^{n'-1}R_n(7)\Big|^2\Big)
\\
& =E\Big(\sup_{n'\in\{1,\ldots,n_T\}}\Big|\sum_{n=0}^{n'-1} \sum_{l=1}^m \mathbbm{1}_{\{N_n\geq 2\}} \int_{t_n}^{t_{n+1}} \sum_{i_0\neq j_0} \int_{t_n}^s\Big(\sigma_{(l)}(X(u),j_0)
\\
&\qquad-\sigma_{(l)}(X(u),i_0)\Big)d[M_{i_0j_0}](u)dW_l(s)\Big|^2\Big)
\\
&\leq  C  E\Big(\sum_{n=0}^{n_T-1}\sum_{l=1}^m \mathbbm{1}_{\{N_n\geq 2\}} \int_{t_n}^{t_{n+1}} E\Big(\Big| \sum_{i_0\neq j_0} \int_{t_n}^s |\sigma_{(l)}(X(u),j_0)
\\
&\qquad -\sigma_{(l)}(X(u),i_0)|d[M_{i_0j_0}](u)\Big|^2 \Big|\mathscr{F}_T^\alpha\Big) ds
\\
&\leq  C  E\Big( \sum_{n=0}^{n_T-1} \mathbbm{1}_{\{N_n\geq 2\}} \int_{t_n}^{t_{n+1}} E\Big( \big(1+\sup_{0 \leq u \leq T}|X(u)|^2\big) 
\\
& \quad \Big(\sum_{i_0\neq j_0} \big([M_{i_0j_0}](s)-[M_{i_0j_0}](t_n)\big)\Big)^2  \Big|\mathscr{F}_T^\alpha\Big) ds
\\
&\leq  C  E\Big( \sum_{n=0}^{n_T-1}\mathbbm{1}_{\{N_n\geq 2\}} \int_{t_n}^{t_{n+1}}  \Big(\sum_{i_0\neq j_0} \big([M_{i_0j_0}](s)-[M_{i_0j_0}](t_n)\big)\Big)^2 
\\
&\qquad \times \Big(1+E \Big(\sup_{0 \leq u \leq T}|X(u)|^2   \Big|\mathscr{F}_T^\alpha\Big)\Big)\Big) ds
\end{align*}
which due to Lemme \ref{l true moment and 1/2 rate} and Lemma \ref{l EN_i, EN_i^2 } gives 
\begin{align*}
E\Big(&\sup_{n'\in\{1,\ldots,n_T\}}\Big|\sum_{n=0}^{n'-1}R_n(7)\Big|^2\Big)
\\
 & \leq C h E\Big(\sum_{n=0}^{n_T-1} \mathbbm{1}_{\{N_n\geq 2\}}   N_n^2  \Big) \leq Ch\sum_{n=0}^{n_T-1} \sum_{N=0}^\infty \mathbbm{1}_{\{N\geq 2\}}   N^2P(N_n=N)
 \\
 & \leq Ch \sum_{n=0}^{n_T-1}\sum_{N=2}^\infty    N^2P(N_n\geq N) 
\\
& \leq Ch \sum_{n=0}^{n_T-1}\sum_{N=0}^\infty   ( N+2)^2 h^{N+2} q^{N+2} \leq Ch^3 n_T\sum_{N=0}^\infty   ( N+2)^2 \frac{1}{2^N} \leq Ch^2 .
\end{align*}
One again notices that $\{\sum_{n=0}^{n'-1}R_n(8);n'\in\{1,\ldots,n_T\}\}$ is a square integrable martingale with respect to filtration $\{\mathscr{F}_T^\alpha\vee\mathscr{F}_{t_{n'}}^W;n'\in\{1,\ldots,n_T\}\}$. By  Burkholder-Davis-Gundy inequality, H\"older's inequality,  Remark \ref{rem:growth} and Lemma \ref{l true moment and 1/2 rate}, one obtains, 
\begin{align*}
E\Big(&\sup_{n'\in\{1,\ldots,n_T\}}\Big|\sum_{n=0}^{n'-1}R_n(8)\Big|^2\Big)  
\\
& =E\Big(\sup_{n'\in\{1,\ldots,n_T\}}\Big| \sum_{n=0}^{n'-1} \sum_{l=1}^m \int_{t_n}^{t_{n+1}} \sum_{i_0\neq j_0} \int_{t_n}^s\Big(\sigma_{(l)}(X(u),i_0)
\\
&\qquad -\sigma_{(l)}(X(u),j_0)\Big)d\langle M_{i_0j_0}\rangle(u)dW_l(s)\Big|^2\Big)
 \\
 & \leq  C  h E \sum_{n=0}^{n_T-1}\sum_{l=1}^m \int_{t_n}^{t_{n+1}}  \sum_{i_0\neq j_0} \int_{t_n}^s |\sigma_{(l)}(X(u),i_0)-\sigma_{(l)}(X(u),j_0)|^2 
 \\
 &  \qquad \times q_{i_0j_0} \mathbbm{1}_{\{\alpha(u-)=i_0\}} du  ds
 \\
 & \leq  C  h E \sum_{n=0}^{n_T-1}  \int_{t_n}^{t_{n+1}}  \int_{t_n}^s (1+|X(u)|^2) du  ds \leq C h^2.  
\end{align*}
 Finally, it is clear that $\{\sum_{n=0}^{n'-1}R_n(12);n'\in\{1,\ldots,n_T\}\}$ is a square integrable martingale with respect to filtration $\{\mathscr{F}_T^\alpha\vee\mathscr{F}_{t_{n'}};n'\in\{1,\ldots,n_T\}\}$. One uses Burkholder-Davis-Gundy inequality and gets the following estimates, 
\begin{align*}
E&\Big( \sup_{n'\in \{1,\ldots,n_T \} }\Big|\sum^{n'-1}_{n=0}R_n(12) \Big|^2 \Big)
\\
=&E\Big( \sup_{n'\in \{1,\ldots,n_T \} }\Big|\sum^{n'-1}_{n=0}\sum_{l=1}^m \int_{t_n}^{t_{n+1}} \sum_{k=0}^{\nu(t_n,s)}\sum_{l_1=1}^{m} \int_{\tau_k^n}^{\tau_{k+1}^n}\Big( \mathcal{D}\sigma_{(l)}(X(\tau_k^n),\alpha(\tau_k^n))  
\\
&\times \sigma_{(l_1)}(X(u),\alpha(u)) -\mathcal{D}\sigma_{(l)}(X_n,\alpha_n)\sigma_{(l_1)}(X_n,\alpha_n) \Big) dW_{l_1}(u)dW_l(s)\Big|^2 \Big)
 \\
\leq & CE \sum^{n_T-1}_{n=0}\sum_{l,l_1=1}^m \int_{t_n}^{t_{n+1}} \sum_{k=0}^{\nu(t_n,s)}E\Big( \Big|\int_{\tau_k^n}^{\tau_{k+1}^n}\Big( \mathcal{D}\sigma_{(l)}(X(\tau_k^n),\alpha(\tau_k^n)) \sigma_{(l_1)}(X(u),\alpha(u)) 
\\
&-\mathcal{D}\sigma_{(l)}(X_n,\alpha_n)\sigma_{(l_1)}(X_n,\alpha_n) \Big) dW_{l_1}(u)\Big|^2\Big| \mathscr{F}_T^{\alpha} \Big) ds 
\\
\leq&CE \sum^{n_T-1}_{n=0}\sum_{l,l_1=1}^m \int_{t_n}^{t_{n+1}} \sum_{k=0}^{\nu(t_n,s)}E\Big( \int_{\tau_k^n}^{\tau_{k+1}^n}| \mathcal{D}\sigma_{(l)}(X(\tau_k^n),\alpha(u)) \sigma_{(l_1)}(X(u),\alpha(u)) \\
&-\mathcal{D}\sigma_{(l)}(X_n,\alpha_n)\sigma_{(l_1)}(X_n,\alpha_n) |^2 du\Big| \mathscr{F}_T^{\alpha} \Big) ds \\
\leq&CE \sum^{n_T-1}_{n=0}\sum_{l,l_1=1}^m \int_{t_n}^{t_{n+1}} \sum_{k=0}^{\nu(t_n,s)}E\Big( \int_{\tau_k^n}^{\tau_{k+1}^n}| \mathcal{D}\sigma_{(l)}(X(\tau_k^n),\alpha(u)) \sigma_{(l_1)}(X(u),\alpha(u)) \\
&-\mathcal{D}\sigma_{(l)}(X(u),\alpha(u)) \sigma_{(l_1)}(X(u),\alpha(u)) |^2 du\Big| \mathscr{F}_T^{\alpha} \Big) ds \\
+&CE \sum^{n_T-1}_{n=0}\sum_{l,l_1=1}^m \int_{t_n}^{t_{n+1}} \sum_{k=0}^{\nu(t_n,s)}E\Big( \int_{\tau_k^n}^{\tau_{k+1}^n}| \mathcal{D}\sigma_{(l)}(X(u),\alpha(u)) \sigma_{(l_1)}(X(u),\alpha(u))\\
&-\mathcal{D}\sigma_{(l)}(X_n,\alpha_n)\sigma_{(l_1)}(X_n,\alpha_n) |^2 du\Big| \mathscr{F}_T^{\alpha} \Big) ds \\
\leq&CE \sum^{n_T-1}_{n=0}\sum_{l,l_1=1}^m \int_{t_n}^{t_{n+1}} \sum_{k=0}^{\nu(t_n,s)}E\Big( \int_{\tau_k^n}^{\tau_{k+1}^n}| \mathcal{D}\sigma_{(l)}(X(\tau_k^n),\alpha(u))  -\mathcal{D}\sigma_{(l)}(X(u),\alpha(u))|^2\\
&\times|\sigma_{(l_1)}(X(u),\alpha(u)) |^2 du\Big| \mathscr{F}_T^{\alpha} \Big) ds \\
+&CE \sum^{n_T-1}_{n=0}\sum_{l,l_1=1}^m \int_{t_n}^{t_{n+1}}\mathbf{1}_{\{N_n=0\}} E\Big( \int_{t_n}^{s}| \mathcal{D}\sigma_{(l)}(X(u),\alpha(u)) \sigma_{(l_1)}(X(u),\alpha(u))\\
&-\mathcal{D}\sigma_{(l)}(X_n,\alpha_n)\sigma_{(l_1)}(X_n,\alpha_n) |^2 du\Big| \mathscr{F}_T^{\alpha} \Big) ds \\
+&CE \sum^{n_T-1}_{n=0}\sum_{l,l_1=1}^m \int_{t_n}^{t_{n+1}}\mathbf{1}_{\{N_n\geq 1\}} E\Big( \int_{t_n}^{s}| \mathcal{D}\sigma_{(l)}(X(u),\alpha(u)) \sigma_{(l_1)}(X(u),\alpha(u))\\
&-\mathcal{D}\sigma_{(l)}(X_n,\alpha_n)\sigma_{(l_1)}(X_n,\alpha_n) |^2 du\Big| \mathscr{F}_T^{\alpha} \Big) ds 
\end{align*}
Further, on the application of Remark 2.1, Assumption H-3 and H\"older's inequality, one obtains,	
\begin{align*}
E&\Big( \sup_{n'\in \{1,\ldots,n_T \} }\Big|\sum^{n'-1}_{n=0}R_n(12) \Big|^2 \Big)\\
\leq &CE \sum^{n_T-1}_{n=0}\int_{t_n}^{t_{n+1}} \sum_{k=0}^{\nu(t_n,s)}E\Big( \int_{\tau_k^n}^{\tau_{k+1}^n}| X(\tau_k^n) -X(u)|^2|(1+|X(u)|) |^2 du\Big| \mathscr{F}_T^{\alpha} \Big) ds \\
+&CE \sum^{n_T-1}_{n=0}\sum_{l,l_1=1}^m \int_{t_n}^{t_{n+1}}\mathbf{1}_{\{N_n=0\}} E\Big( \int_{t_n}^{s}| \mathcal{D}\sigma_{(l)}(X(u),\alpha_n) \sigma_{(l_1)}(X(u),\alpha_n)\\
&-\mathcal{D}\sigma_{(l)}(X_n,\alpha_n)\sigma_{(l_1)}(X_n,\alpha_n) |^2 du\Big| \mathscr{F}_T^{\alpha} \Big) ds \\
+&CE \sum^{n_T-1}_{n=0} \int_{t_n}^{t_{n+1}}\mathbf{1}_{\{N_n\geq 1\}} E\Big( \int_{t_n}^{s}((1+|X(u)|^2)+(1+|X_n|)^2) du\Big| \mathscr{F}_T^{\alpha} \Big) ds\\
\leq &CE \sum^{n_T-1}_{n=0}\int_{t_n}^{t_{n+1}} \sum_{k=0}^{\nu(t_n,s)} \int_{\tau_k^n}^{\tau_{k+1}^n}\{E(| X(\tau_k^n) -X(u)|^4\big|\mathscr{F}_T^{\alpha})\}^{\frac{1}{2}}
\\
& \times \{E((1+|X(u) |^4)\big|\mathscr{F}_T^{\alpha})\}^{\frac{1}{2}} du ds  \\
& +CE \sum^{n_T-1}_{n=0}\int_{t_n}^{t_{n+1}}\mathbf{1}_{\{N_n=0\}} \int_{t_n}^{s}E\big(|X(u)-X_n|^2\big|\mathscr{F}_T^{\alpha}  \big)duds\\
+&CE \sum^{n_T-1}_{n=0} \int_{t_n}^{t_{n+1}}\mathbf{1}_{\{N_n\geq 1\}} E\Big( \int_{t_n}^{s}((1+|X(u)|^2)+(1+|X_n|)^2) du\Big| \mathscr{F}_T^{\alpha} \Big) ds\\
\end{align*}	
which due to Lemma [4.1, 5.2] gives
\begin{align*}
E&\Big( \sup_{n'\in \{1,\ldots,n_T \} }\Big|\sum^{n'-1}_{n=0}R_n(12) \Big|^2 \Big)\leq  \sum^{n_T-1}_{n=0} h^3+Ch\sum^{n_T-1}\int_{t_n}^{t_{n+1}}E(\mathbf{1}_{\{N_n\geq 1\}})ds\\
& \leq Ch^2
\end{align*}	

\noindent
This completes the proof of the lemma. 
\end{proof}
After proving the necessary lemmas, one now proceeds with the proof of the main result of this article i.e. Theorem \ref{thm:main}. 
\begin{proof}[\textbf{Proof of Theorem \ref{thm:main}}]
Let us recall expansion \eqref{scheme complete derivation}  and  scheme \eqref{eq:scheme} and hence write,  
\begin{align*}
X_{n} & - Y_{n}= X_{0}  - Y_{0}+ \sum_{k=0}^{n-1}\big(b(X_k,\alpha_k)-b(Y_k, \alpha_k)\big)h 
\\
& +\sum_{k=0}^{n-1}\sum_{l=1}^m\big(\sigma_{(l)}(X_k,\alpha_k)-\sigma_{(l)}(Y_k,\alpha_k)\big)\Delta_k W_l
\\
&+\sum_{k=1}^{n-1} \sum_{l,l_1=1}^m\int_{t_k}^{t_{k+1}}  \int_{t_k}^{s} \big( \mathcal{D} \sigma_{(l)}(X_k,\alpha_k) \sigma_{(l_1)}(X_k,\alpha_k) 
\\
& - \mathcal{D} \sigma_{(l)}(Y_k,\alpha_k) \sigma_{(l_1)}(Y_k,\alpha_k)\big)dW_{l_1}(u) dW_l(s) 
\\
&+\sum_{k=1}^{n-1} \sum_{l=1}^m \mathbbm{1}_{\{N_k=1\}}\Big(\sigma_{(l)}(X_k,\alpha_{k+1})-\sigma_{(l)}(X_k,\alpha_k)
\\
& -\sigma_{(l)}(Y_k,\alpha_{k+1})+\sigma_{(l)}(Y_k,\alpha_k)\Big)\Big(W_l(t_{k+1})-W_l(\tau_1^k)\Big)
\\
&+\sum_{k=0}^{n-1}\sum_{i=1}^{12}R_{k}(i)
\end{align*}
which further implies, 
\begin{align}
E&\Big(\sup_{n\in\{1,\ldots,n'\}}|X_{n} - Y_{n}|^2 \Big)\leq  CE|X_{0}  - Y_{0}|^2 \notag
\\
&+CE\Big(\sup_{n\in\{1,\ldots,n'\}} \Big|\sum_{k=0}^{n-1}\big(b(X_k,\alpha_k)-b(Y_k, \alpha_k)\big)h \Big|^2\Big) \notag
\\
&+CE\Big(\sup_{n\in\{1,\ldots,n'\}}\Big|\sum_{k=0}^{n-1}\sum_{l=1}^m\big(\sigma_{(l)}(X_k,\alpha_k)-\sigma_{(l)}(Y_k,\alpha_k)\big)\Delta_k W_l \Big|^2\Big) \notag
\\
&+CE\Big(\sup_{n\in\{1,\ldots,n'\}}\Big| \sum_{k=0}^{n-1} \sum_{l,l_1=1}^m\int_{t_k}^{t_{k+1}}  \int_{t_k}^{s} \big( \mathcal{D} \sigma_{(l)}(X_k,\alpha_k) \sigma_{(l_1)}(X_k,\alpha_k) \notag
\\
& \qquad - \mathcal{D} \sigma_{(l)}(Y_k,\alpha_k) \sigma_{(l_1)}(Y_k,\alpha_k)\big)dW_{l_1}(u) dW_l(s) \Big|^2\Big) \notag
\\
&+CE\Big(\sup_{n\in\{1,\ldots,n'\}}\Big|\sum_{k=0}^{n-1} \sum_{l=1}^m \mathbbm{1}_{\{N_k=1\}}\Big(\sigma_{(l)}(X_k,\alpha_{k+1})-\sigma_{(l)}(X_k,\alpha_k)\notag
\\
& \qquad -\sigma_{(l)}(Y_k,\alpha_{k+1})+\sigma_{(l)}(Y_k,\alpha_k)\Big)  \Big(W_l(t_{k+1})-W_l(\tau_1^k)\Big)\Big|^2 \Big) \notag
\\
&+CE\Big(\sup_{n\in\{1,\ldots,n'\}}\Big| \sum_{k=0}^{n-1}\sum_{i=1}^{12}R_{k}(i)\Big|^2\Big) \notag
\\
& =: CE|X_{0}  - Y_{0}|^2+ S_1+S_2+S_3+S_4 \notag
\\
&+C\sum_{i=1}^{12} E\Big(\sup_{n\in\{1,\ldots,n_T\}}\Big| \sum_{k=0}^{n-1}R_{k}(i)\Big|^2\Big) \label{eq:S1S4}
\end{align}
for any $n'=1,\ldots,n_T$. By using Assumption H-3, one can estimate $S_1$ as follows, 
\begin{align}
S_1&:= CE\Big(\sup_{n\in\{1,\ldots,n'\}} \Big|\sum_{k=0}^{n-1}\big(b(X_k,\alpha_k)-b(Y_k, \alpha_k)\big)h \Big|^2\Big) \notag
\\
& \leq Cn' h^2 E\sum_{k=0}^{n'-1}\big|b(X_k,\alpha_k)-b(Y_k, \alpha_k)\big|^2 \leq C h \sum_{k=0}^{n'-1}E\Big(\sup_{n\in{\{0,\ldots,k\}}}|X_n-Y_n|^2\Big) \label{eq:S1}
\end{align}
for any $n'=0,1,\ldots,n_T$.  For $S_2$, one uses  Burkholder-Davis-Gundy inequality and Assumption H-3 to get the following estimate, 
\begin{align}
S_2 & :=CE\Big(\sup_{n\in\{1,\ldots,n'\}}\Big|\sum_{k=0}^{n-1}\sum_{l=1}^m\big(\sigma_{(l)}(X_k,\alpha_k)-\sigma_{(l)}(Y_k,\alpha_k)\big)\Delta_k W_l\Big|^2\Big) \notag
\\
& \leq C h E\Big(\sum_{k=0}^{n'-1}\sum_{l=1}^m\big|\sigma_{(l)}(X_k,\alpha_k)-\sigma_{(l)}(Y_k,\alpha_k)\big|^2 \Big)  \notag
\\
& \leq C h \sum_{k=0}^{n'-1}E\Big(\sup_{n\in{\{0,\ldots,k\}}}|X_n-Y_n|^2\Big)\label{eq:S2}
\end{align}
for any $n'=1,\ldots,n_T$. Due to Burkholder-Davis-Gundy inequality, $S_3$ can be estimated by, 
\begin{align}
S_3 & := CE\Big(\sup_{n\in\{1,\ldots,n'\}}\Big| \sum_{k=0}^{n-1} \sum_{l,l_1=1}^m\int_{t_k}^{t_{k+1}}  \int_{t_k}^{s} \big( \mathcal{D} \sigma_{(l)}(X_k,\alpha_k) \sigma_{(l_1)}(X_k,\alpha_k) \notag
\\
& \qquad - \mathcal{D} \sigma_{(l)}(Y_k,\alpha_k) \sigma_{(l_1)}(Y_k,\alpha_k)\big)dW_{l_1}(u) dW_l(s) \Big|^2\Big) \notag
\\
& \leq C  E\Big( \sum_{k=0}^{n'-1}\sum_{l,l_1=1}^m\int_{t_k}^{t_{k+1}}  E\Big(\Big|  \int_{t_k}^{s} \big( \mathcal{D} \sigma_{(l)}(X_k,\alpha_k) \sigma_{(l_1)}(X_k,\alpha_k) \notag
\\
& \qquad - \mathcal{D} \sigma_{(l)}(Y_k,\alpha_k) \sigma_{(l_1)}(Y_k,\alpha_k)\big)dW_{l_1}(u) \Big|^2 \Big| \mathscr{F}_T^\alpha \vee \mathscr{F}_{t_k}\Big)  ds \Big) \notag
\\
& \leq C  E\Big( \sum_{k=0}^{n'-1}\sum_{l,l_1=1}^m\int_{t_k}^{t_{k+1}}    \int_{t_k}^{s} E\Big( \big| \mathcal{D} \sigma_{(l)}(X_k,\alpha_k) \sigma_{(l_1)}(X_k,\alpha_k) \notag
\\
& \qquad - \mathcal{D} \sigma_{(l)}(Y_k,\alpha_k) \sigma_{(l_1)}(Y_k,\alpha_k)\big|^2 du \Big| \mathscr{F}_T^\alpha \vee \mathscr{F}_{t_k}\Big)  ds \Big) \notag
\end{align}
which on using Assumption H-3 gives the following, 
\begin{align}
S_3\leq C  h^2  \sum_{k=0}^{n'-1} E|X_k  - Y_k|^2  \leq C h  \sum_{k=0}^{n'-1} E\Big(\sup_{n\in\{0,\ldots,k\}}|X_n  - Y_n|^2\Big)  \label{eq:S3} 
\end{align}
for any $n'=1,\ldots,n_T$. For estimating $S_4$, notice that 
\begin{align*}
& \Big\{\sum_{k=0}^{n-1} \sum_{l=1}^m \mathbbm{1}_{\{N_k=1\}}\Big(\sigma_{(l)}(X_k,\alpha_{k+1})-\sigma_{(l)}(X_k,\alpha_k)-\sigma_{(l)}(Y_k,\alpha_{k+1})+\sigma_{(l)}(Y_k,\alpha_k)\Big) \notag
\\
& \qquad \times \Big(W_l(t_{k+1})-W_l(\tau_1^k)\Big) ; n\in\{1,\ldots,n_T\} \Big\}
\end{align*}
is a square integrable martingale with respect to filtration $\{\mathscr{F}_T^\alpha \vee \mathscr{F}_{t_{n}}^W;n\in\{1,\ldots,n_T\}\}$ and hence due to  Burkholder-Davis-Gundy inequality, one obtains 
\begin{align}
S_4& :=CE\Big(\sup_{n\in\{1,\ldots,n'\}}\Big|\sum_{k=0}^{n-1} \sum_{l=1}^m \mathbbm{1}_{\{N_k=1\}}\Big(\sigma_{(l)}(X_k,\alpha_{k+1}) \notag
\\
& \qquad -\sigma_{(l)}(X_k,\alpha_k)-\sigma_{(l)}(Y_k,\alpha_{k+1})+\sigma_{(l)}(Y_k,\alpha_k)\Big)\Big(W_l(t_{k+1})-W_l(\tau_1^k)\Big)\Big|^2 \Big) \notag
\\
& \leq CE\Big(\sum_{k=0}^{n'-1} \sum_{l=1}^m \mathbbm{1}_{\{N_k=1\}} \big|\sigma_{(l)}(X_k,\alpha_{k+1})-\sigma_{(l)}(X_k,\alpha_k)\notag
\\
& \qquad -\sigma_{(l)}(Y_k,\alpha_{k+1})+\sigma_{(l)}(Y_k,\alpha_k)\big|^2  \big|W_l(t_{k+1})-W_l(\tau_1^k) \big|^2 \Big) \notag
\\
& \leq CE\Big(\sum_{k=0}^{n'-1} \sum_{l=1}^m \mathbbm{1}_{\{N_k=1\}} \big(\big|\sigma_{(l)}(X_k,\alpha_{k+1})- \sigma_{(l)}(Y_k,\alpha_{k+1})\big|^2  \notag
\\
& \qquad + \big|\sigma_{(l)}(X_k,\alpha_k)-\sigma_{(l)}(Y_k,\alpha_k)\big|^2\big) E\big(\big|W_l(t_{k+1})-W_l(\tau_1^k) \big|^2\big|\mathscr{F}_T^\alpha \vee \mathscr{F}_{\tau_1^k \wedge t_{k}}\big) \Big) \notag
\\
&\leq C h \sum_{k=0}^{n'-1}E\Big(\sup_{n\in{\{0,\ldots,k\}}}|X_n-Y_n|^2\Big)   \label{eq:S4}
\end{align}
for any $n'=1,\ldots,n_T$. On substituting estimates from \eqref{eq:S1} to \eqref{eq:S4} and Lemma \ref{lem:R3R10}, Lemma \ref{lem:R4R5R11}, Lemma \ref{lem:R2R9} and Lemma \ref{lem:R1R6R7R8R12} in \eqref{eq:S1S4}, one obtains the following  estimates, 
\begin{align*}
E&\Big(\sup_{n\in\{0,\ldots,n'\}}|X_{n} - Y_{n}|^2 \Big)\leq CE|X_{0}  - Y_{0}|^2+Ch^2
\\
& +C h \sum_{k=0}^{n'-1}E\Big(\sup_{n\in{\{0,\ldots,k\}}}|X_n-Y_n|^2\Big)
\end{align*}
for any $n'=1,\ldots,n_T$. The Gronwall's lemma and Assumption H-1 completes the proof. 
\end{proof}


\subsection*{Acknowledgement}
First author gratefully acknowledge financial support provided  by  Science and Engineering Research Board (SERB) under its MATRICS  program through grant number SER-1329-MTD.

\bibliographystyle{amsplain}

\end{document}